\documentclass[11pt,a4paper]{article}
\usepackage[]{geometry}
\usepackage[mathcal]{euscript}
\usepackage{bm,url}                     
\usepackage{amsfonts}
\usepackage{amssymb}
\usepackage{amsmath}
\usepackage{amsthm}
\usepackage[utf8x]{inputenc}
\usepackage{float}
\usepackage[pdftex]{graphicx}
\DeclareGraphicsExtensions{.bmp,.png,.pdf,.jpg,.ps}
\usepackage{natbib}                 
\usepackage{colortbl}  
\usepackage{todonotes}             
\usepackage{booktabs}
\newtheorem{definition}{Definition}
\newtheorem{remark}{Remark}
\newtheorem{theorem}{Theorem}
\newtheorem{lemma}{Lemma}
\newtheorem{proposition}{Proposition}
\newtheorem{corollary}{Corollary}

\newtheorem{notation}[theorem]{Notation}

\usepackage[english]{babel}
\usepackage{amsmath}
\usepackage{amsfonts}
\usepackage{amssymb}
\usepackage{graphicx}
\usepackage{color}
\usepackage{array}
\usepackage{mathrsfs}
\usepackage{hyperref}

\definecolor{violet}{rgb}{0.7,0,0.6}
\definecolor{OliveGreen}{RGB}{85,107,47}

\title{Fractional iterated Ornstein-Uhlenbeck Processes}
\author{Juan Kalemkerian\\
Universidad de la República, Facultad de Ciencias.}

\begin{document}
\maketitle
\begin{abstract}
\noindent In this work we present a Gaussian process that arise from the iteration of p
fractional Ornstein-Uhlenbeck processes generated by the same fractional Brownian
motion. This iteration results, when the values of lambdas are pairwise differents, in a particular linear combination of those processes. Although for $H>1/2$ each term of 
the linear combination is a long memory processes, we prove that it results in a short memory 
processes. We include applications to real data that show improvement in predictive performance 
compared with different ARMA models.
\end{abstract}

AMS: 62M10

\newpage

\section{Introduction}

We begin with the following definition of fractional Brownian motion. 
\begin{definition}
A fractional Brownian motion with Hurst parameter $H\in \left( 0,1\right] $,
is an almost surely continuous centered Gaussian process $\left\{
B_{H}(t)\right\} _{t\in \mathbb{R}}$ with 
\begin{equation*}
\mathbb{E}\left( B_{H}(t)B_{H}(s)\right) =\frac{1}{2}\left( \left\vert
t\right\vert ^{2H}+\left\vert s\right\vert ^{2H}-\left\vert t-s\right\vert
^{2H}\right) ,\text{ \ }t,s\in \mathbb{R}.
\end{equation*}
\end{definition}

When $H=1/2,$ fractional Brownian motion become in a standard Brownian
motion.  An Ornstein-Uhlenbeck process is a
Gaussian process defined by $X_{t}=\sigma \int_{-\infty }^{t}e^{-\lambda
\left( t-s\right) }dB_{1/2}(t)$ for $t\in \mathbb{R},$ where $\sigma
,\lambda >0,$ are parameters (Ornstein \& Uhlenbeck, 1930). This process is the unique stationary solution of the
Langevin equation (Langevin, 1908), defined by \[dX_{t}=-\lambda
X_{t}+\sigma dB_{1/2}(t).\] 
If we consider the Langevin equation with a fractional Brownian
motion, this is $dX_{t}=-\lambda X_{t}+\sigma dB_{H}(t),$ then $X_{t}=\sigma
\int_{-\infty }^{t}e^{-\lambda \left( t-s\right) }dB_{H}(t)$ for $t\in 
\mathbb{R}$ is the unique stationary solution (Cheridito et al, 2003). In this work, we use the
notation $\left\{ X_{t}\right\} _{t\in \mathbb{R}}\sim $FOU$\left( \lambda
,\sigma ,H\right) $, for any process defined as $X_{t}=\sigma \int_{-\infty
}^{t}e^{-\lambda \left( t-s\right) }dB_{H}(t)$, where $\sigma ,\lambda >0,$ $%
H\in \left( 0,1\right] $.\\ If we change the process $\left\{ B_{H}(t)\right\}
_{t\in \mathbb{R}}$ by another $\left\{ y(t)\right\} _{t\in \mathbb{R}}$ we
can define the operators $T_{\lambda }(y)(t):=\int_{-\infty }^{t}e^{-\lambda
(t-s)}dy(s)$ and$\ $for each $h=0,1,2,...$

\begin{equation}
T_{\lambda }^{(h)}(y)(t):=\int_{-\infty }^{t}e^{-\lambda (t-s)}\frac{\left(
-\lambda \left( t-s\right) \right) ^{h}}{h!}dy(s).  \label{hh}
\end{equation}

These transformations are called, $OU$
operator with parameter $\lambda $ and $OU$ operator of degree $h$ and
parameter $\lambda $ respectively (Arratia et al, 2016).

Observe that $T_{\lambda }^{(0)}=T_{\lambda }.$

Given $\left\{ B_{H}(s)\right\} _{s\in \mathbb{R}}$  a fractional Brownian
motion with parameter $H$, and $\lambda _{1}\neq \lambda _{2}$ are real positive numbers, we define
the processes \\$%
X_{t}^{(i)}:=T_{\lambda _{i}}\left(\sigma B_{H}\right) (t)= \sigma \int_{-\infty
}^{t}e^{-\lambda _{i}(t-s)}dB_{H}(s)$ for $i=1,2.$ This is \\
$\left\{ X_{t}^{(i)}\right\} _{t\in \mathbb{R}}\sim $FOU$\left( \lambda_{i}
,\sigma ,H\right) $ for $i=1,2$ generated by the same fractional Brownian motion.
It can be proved that the process defined as $X_{t}:=\left( T_{\lambda
_{1}}\circ T_{\lambda _{2}}\right) \left( B_{H}\right) (t)$ is equal to $X_{t}=%
\frac{\lambda _{1}}{\lambda _{1}-\lambda _{2}}X_{t}^{\left( 1\right) }+\frac{%
\lambda _{2}}{\lambda _{2}-\lambda _{1}}X_{t}^{\left( 2\right) },$ this is a
particular linear combination of process $\left\{ X_{t}^{(1)}\right\}
_{t\in \mathbb{R}}$ and $\left\{ X_{t}^{(2)}\right\} _{t\in \mathbb{R}}.$

This implies that $\left( T_{\lambda _{1}}\circ T_{\lambda
_{2}}\right) \left( B_{H}\right) =\left( T_{\lambda _{2}}\circ T_{\lambda
_{1}}\right) \left( B_{H}\right) $.

In general,  if we compose 
$p$ times the operator $T_{\lambda }$ result in the following equality: $%
T_{\lambda }^{p}=\sum_{j=0}^{p-1}\binom{p-1}{j} T_{\lambda
}^{\left( j\right) }.$ And if we compose $p_{1}$ times the operator $%
T_{\lambda _{1}},$ $p_{2}$ times $T_{\lambda _{2}},...,$ and $p_{q}$ times
the operator $T_{\lambda _{q}},$ for $\lambda _{i}$ pairwise different, we have  

\begin{equation*}
\prod\limits_{i=1}^{q}T_{\lambda _{i}}^{p_{i}}=\sum_{i=1}^{q}K_{i}\left(
\lambda \right) T_{\lambda _{i}}^{p_{i}}=\sum_{i=1}^{q}K_{i}\left( \lambda
\right) \sum_{j=0}^{p_{i}-1}\binom{p_{i}-1}{j} T_{\lambda
_{i}}^{\left( j\right) }
\end{equation*}%
where $%
p=p_{1}+p_{2}+...+p_{q}$ and 

\begin{equation}
K_{i}\left( \lambda \right) =K_{i}\left( \lambda _{1},\lambda
_{2},...,\lambda _{q}\right) =\frac{1}{\prod\limits_{j\neq i}\left(
1-\lambda _{j}/\lambda _{i}\right) }  \label{k_i}
\end{equation}

(Arratia et al 2016).

It is known that for $H>1/2$ every FOU$\left( \lambda ,\sigma ,H\right) $ is a long
memory process (Cheridito et al, 2013), this is $\sum_{n=-\infty }^{+\infty }\left\vert \gamma
\left( n\right) \right\vert =+\infty $ where $\gamma \left( n\right) =%
\mathbb{E}\left( X_{0}X_{n}\right) .$ In this work we prove in section 2 that if we
compose at least two operators of the form $T_{\lambda }$ evaluated in a
fractional Brownian motion, with Hurst parameter $H>1/2$, we obtain a
process $\left\{ X_{t}\right\} _{t\in \mathbb{R}}$ such that $%
\sum_{n=-\infty }^{+\infty }\left\vert \mathbb{E}\left( X_{0}X_{n}\right)
\right\vert <+\infty.$  Furhter, the process obtained
has short memory. In section 2, we define a FOU$\left(
p\right) $ processes, and summarize the results needed to obtain the auto-covariance function. 
We also obtain its spectral
density and deduce that in the case in wich $p\geq 2$ it is a short memory
process. In section 3, we apply
these models to  real data sets and compare the performance of these
models with ARMA models according to their predictive power. In section 4, we make 
the demonstration of the results established in section 2. Our concluding remarks are in
section 5.

\section{Definitions and properties}

\bigskip We start with the definition of the  fractional iterated Ornstein-Uhlenbeck process. 

\begin{definition}
If $\left\{ \sigma B_{H}(s)\right\} _{s\in \mathbb{R}}$ is a fractional
Brownian motion with Hurst parameter $H,$ and escale parameter $\sigma ,$
and a pairwise different real positive numbers $\lambda _{1},\lambda
_{2},...,\lambda _{q}$ and $p_{1},p_{2},...,p_{q}\in \mathbb{N}$ such that $%
p_{1}+p_{2}+...+p_{q}=p$, we define $\left\{ X_{t}\right\} _{t\in \mathbb{R}%
} $ by 
\begin{equation*}
X_{t}:=\prod\limits_{i=1}^{q}T_{\lambda _{i}}^{p_{i}}(\sigma
B_{H})(t)=\sum_{i=1}^{q}K_{i}\left( \lambda \right)
\sum_{j=0}^{p_{i}-1}\binom{p_{i}-1}{j} T_{\lambda _{i}}^{\left(
j\right) }(\sigma B_{H})(t)
\end{equation*}%
where the numbers $K_{i}\left( \lambda \right) $ and the operators $%
T_{\lambda _{i}}^{\left( j\right) }$ were defined in (\ref{k_i}) and (\ref%
{hh}) respectively.
\end{definition}

\begin{notation}
$\left\{ X_{t}\right\} _{t\in \mathbb{R}}\sim FOU\left( \lambda _{1}^{\left(
p_{1}\right) },\lambda _{2}^{\left( p_{2}\right) },...,\lambda _{q}^{\left(
p_{q}\right) },\sigma ,H\right), $ or more simply
$\left\{ X_{t}\right\} _{t\in \mathbb{R}}\sim$FOU$(p)$.

\end{notation}

Observe that the notation FOU$\left( \lambda _{1}^{\left( p_{1}\right)
},\lambda _{2}^{\left( p_{2}\right) },...,\lambda _{q}^{\left( p_{q}\right)
},\sigma ,H\right) $ implies that the $\lambda _{i}$ parameters are pairwise
different.

\begin{remark}
When $p_{1}=p_{2}=...=p_{q}=1$ \ the process is equal to
\end{remark}

\begin{equation*}
X_{t}=\prod\limits_{i=1}^{q}T_{\lambda _{i}}(\sigma
B_{H})(t)=\sum_{i=1}^{q}K_{i}\left( \lambda \right) T_{\lambda _{i}}(\sigma
B_{H})(t)
\end{equation*}%
and we call $\left\{ X_{t}\right\} _{t\in \mathbb{R}}\sim $FOU$\left(
\lambda _{1},\lambda _{2},...,\lambda _{q},\sigma ,H\right) .$

\begin{remark}
When $p=1$, we obtain a fractional Ornstein-Uhlenbeck process (FOU$\left(
\lambda ,\sigma ,H\right) $).
\end{remark}

\begin{remark}
Any FOU$\left( \lambda _{1}^{\left( p_{1}\right) },\lambda _{2}^{\left(
p_{2}\right) },...,\lambda _{q}^{\left( p_{q}\right) },\sigma ,H\right) $,
is Gaussian, centered and almost surely continuous process.
\end{remark}

\bigskip Now, we compute the auto-covariance function of any FOU$\left( p\right) $
process. For this we need the following formula, whose proof can be seen in  (Pipiras \& Taqqu, 2000): 
if $H\in \left( 1/2,1\right) $ and \[f,g\in \left\{ f:\mathbb{R\rightarrow R}\text{: }\int \int_{\mathbb{R}%
^{2}}\left\vert f(u)f(v)\right\vert \left\vert u-v\right\vert
^{2H-2}dudv<+\infty \right\}, \] then 
\begin{equation}
\mathbb{E}\left( \int_{-\infty }^{+\infty }f(u)dB_{H}(u)\int_{-\infty
}^{+\infty }g(v)dB_{H}(v)\right) =  \label{pipiras}
\end{equation}%
\begin{equation*}
H(2H-1)\int_{-\infty }^{+\infty }f(u)du\int_{-\infty }^{+\infty
}g(v)\left\vert u-v\right\vert ^{2H-2}dv.
\end{equation*}

We start with $\gamma (t)=\mathbb{E}\left( X_{t}X_{0}\right) =$

\begin{equation*}
\mathbb{E}\sum_{h=1}^{q}K_{h}\left( \lambda \right)
\sum_{j=0}^{p_{h}-1}\binom{p_{h}-1}{j} T_{\lambda _{h}}^{\left(
j\right) }(\sigma B_{H})(t)\sum_{h^{\prime }=1}^{q}K_{h^{\prime }}\left(
\lambda \right) \sum_{j^{\prime }=0}^{p_{h^{\prime }}-1}\binom{p_{h^{\prime }}-1}{j^{\prime }} T_{\lambda _{h^{\prime }}}^{\left( j^{\prime
}\right) }(\sigma B_{H})(0)=
\end{equation*}

\begin{equation*}
\sum_{i=1}^{q}\sum_{i^{\prime }=1}^{q}\sum_{j=0}^{p_{i}-1}\sum_{j^{\prime
}=0}^{p_{i^{\prime }}-1}K_{i}\left( \lambda \right) \binom{p_{i}-1}{j} K_{i^{\prime }}\left( \lambda \right) \binom
{p_{i^{\prime }}-1}{j^{\prime }} \mathbb{E}T_{\lambda _{i}}^{\left(
j\right) }(\sigma B_{H})(t)T_{\lambda _{i^{\prime }}}^{\left( j^{\prime
}\right) }(\sigma B_{H})(0).
\end{equation*}

Define $\gamma _{\lambda ,\lambda ^{\prime }}^{\left( j,j^{\prime }\right)
}(t):=\mathbb{E}T_{\lambda }^{\left( j\right) }(\sigma B_{H})(t)T_{\lambda
^{\prime }}^{\left( j^{\prime }\right) }(\sigma B_{H})(0),$ then

\begin{equation}
\gamma (t)=\sum_{i,i^{\prime }=1}^{q}K_{i}\left( \lambda \right)
K_{i^{\prime }}\left( \lambda \right) \sum_{j=0}^{p_{i}-1}\sum_{j^{\prime
}=0}^{p_{i^{\prime }}-1}\binom{p_{i}-1}{j} \binom {p_{i^{\prime }}-1}{j^{\prime
}} \gamma _{\lambda _{i},\lambda _{i^{\prime
}}}^{\left( j,j^{\prime }\right) }(t).  \label{covgeneral}
\end{equation}

Now, we compute $\gamma _{\lambda ,\lambda ^{\prime }}^{\left( j,j^{\prime
}\right) }(t).$

\begin{equation*}
\bigskip \gamma _{\lambda ,\lambda ^{\prime }}^{\left( j,j^{\prime }\right)
}(t)=\mathbb{E}T_{\lambda }^{\left( j\right) }(\sigma B_{H})(t)T_{\lambda
^{\prime }}^{\left( j^{\prime }\right) }(\sigma B_{H})(0)=
\end{equation*}

\begin{equation}
\sigma ^{2}\mathbb{E}\int_{-\infty }^{t}e^{-\lambda (t-u)}\frac{\left(
-\lambda \left( t-u\right) \right) ^{j}}{j!}dB_{H}(u)\int_{-\infty
}^{0}e^{\lambda ^{\prime }v}\frac{\lambda ^{\prime j^{\prime }}v^{j^{\prime
}}}{j^{\prime }!}dB_{H}(v).  \label{7}
\end{equation}

Using (\ref{pipiras}), we obtain that \ (\ref{7}) is equal to

\begin{equation*}
\sigma ^{2}H(2H-1)\int_{-\infty }^{t}e^{-\lambda (t-u)}\frac{\left( -\lambda
\left( t-u\right) \right) ^{j}}{j!}du\int_{-\infty }^{0}e^{\lambda ^{\prime
}v}\frac{\lambda ^{\prime j^{\prime }}v^{j^{\prime }}}{j^{\prime }!}%
\left\vert u-v\right\vert ^{2H-2}dv=
\end{equation*}

\begin{equation*}
\frac{\sigma ^{2}H(2H-1)\lambda ^{j}\lambda ^{\prime j^{\prime }}}{%
j!j^{\prime }!}\int_{-\infty }^{t}e^{-\lambda (t-u)}\left( u-t\right)
^{j}du\int_{-\infty }^{0}e^{\lambda ^{\prime }v}v^{j^{\prime }}\left\vert
u-v\right\vert ^{2H-2}dv\overset{w=u-t}{=}
\end{equation*}%
\begin{equation*}
\frac{\sigma ^{2}H(2H-1)\lambda ^{j}\lambda ^{\prime j^{\prime }}}{%
j!j^{\prime }!}\int_{-\infty }^{0}e^{\lambda w}w^{j}dw\int_{-\infty
}^{0}e^{\lambda ^{\prime }v}v^{j^{\prime }}\left\vert w+t-v\right\vert
^{2H-2}dv=
\end{equation*}%
\begin{equation}
\frac{\sigma ^{2}H(2H-1)\lambda ^{j}\lambda ^{\prime j^{\prime }}\left(
-1\right) ^{j+j^{\prime }}}{j!j^{\prime }!}\int_{0}^{+\infty }e^{-\lambda
w}w^{j}dw\int_{0}^{+\infty }e^{-\lambda ^{\prime }v}v^{j^{\prime
}}\left\vert v+t-w\right\vert ^{2H-2}dv.  \label{gamma}
\end{equation}

To obtain the results in this work, we need to define the following
functions:%
\begin{equation}
f_{H}^{\left( 1\right) }(x):=e^{-x}\left( \Gamma \left( 2H\right)
-\int_{0}^{x}e^{s}s^{2H-1}ds\right) ,  \label{fH1}
\end{equation}

\begin{equation}
f_{H}^{\left( 2\right) }(x):=e^{x}\left( \Gamma \left( 2H\right)
-\int_{0}^{x}e^{-s}s^{2H-1}ds\right) ,  \label{fH2}
\end{equation}%
\begin{equation*}
f_{H}(x):=f_{H}^{\left( 1\right) }(x)+f_{H}^{\left( 2\right) }(x).
\end{equation*}%
Then, 
\begin{equation}
f_{H}(x):=e^{-x}\left( \Gamma \left( 2H\right)
-\int_{0}^{x}e^{s}s^{2H-1}ds\right) +e^{x}\left( \Gamma \left( 2H\right)
-\int_{0}^{x}e^{-s}s^{2H-1}ds\right) .  \label{fH}
\end{equation}

\begin{remark}
For $H>1/2,$ is verified that $f_{H}^{\left( 1\right) }(x)\rightarrow
+\infty $ and $f_{H}^{\left( 2\right) }(x)\rightarrow -\infty $ when $%
x\rightarrow +\infty.$
\end{remark}

As $H$ increases, the functions $f_{H}$ increases as can be seen in Figure \ref{f_H}. 
Then, when $x \rightarrow +\infty$, it is verified that as $H$ increases the functions $f_{H}$ goes to zero more slowly.
The following proposition includes properties of $f_{H}$ that will be used
later.
We denote $f \sim g$ for $x \rightarrow a$, when $f(x)/g(x) \rightarrow 1$ for $x \rightarrow a$.

\begin{proposition}
If $H>1/2,$ $\alpha ,\beta >0,$ then
\end{proposition}

\begin{enumerate}
\item $\alpha ^{1-2H}f_{H}^{\left( 1\right) }(\alpha x)+\beta
^{1-2H}f_{H}^{\left( 2\right) }(\beta x)\rightarrow 0$ when $x\rightarrow
+\infty .$

\item $\alpha ^{1-2H}f_{H}^{\left( 1\right) }(\alpha x)+\beta
^{1-2H}f_{H}^{\left( 2\right) }(\beta x)\sim \frac{\alpha +\beta }{\alpha
\beta }\left( 2H-1\right) x^{2H-2}$ when $x\rightarrow +\infty .$

\item $f_{H}(x)\sim 2(2H-1)x^{2H-2}$\ when $x\rightarrow +\infty .$

\item $f_{H}(x)-f_{H}(0)=f_{H}(x)-2\Gamma \left( 2H\right) =-\frac{x^{2H}}{H}%
+o\left( x^{2H}\right) $ when $x\rightarrow 0.$

\item $f_{H}(x)=\frac{\Gamma \left( 2H+1\right) \sin \left( H\pi \right) }{%
2\pi }\int_{-\infty }^{+\infty }\frac{e^{iv}\left\vert v\right\vert ^{1-2H}}{%
v^{2}+x^{2}}dv.$
\end{enumerate}
\begin{figure}[H]
 \centering
   \includegraphics[scale=0.43]{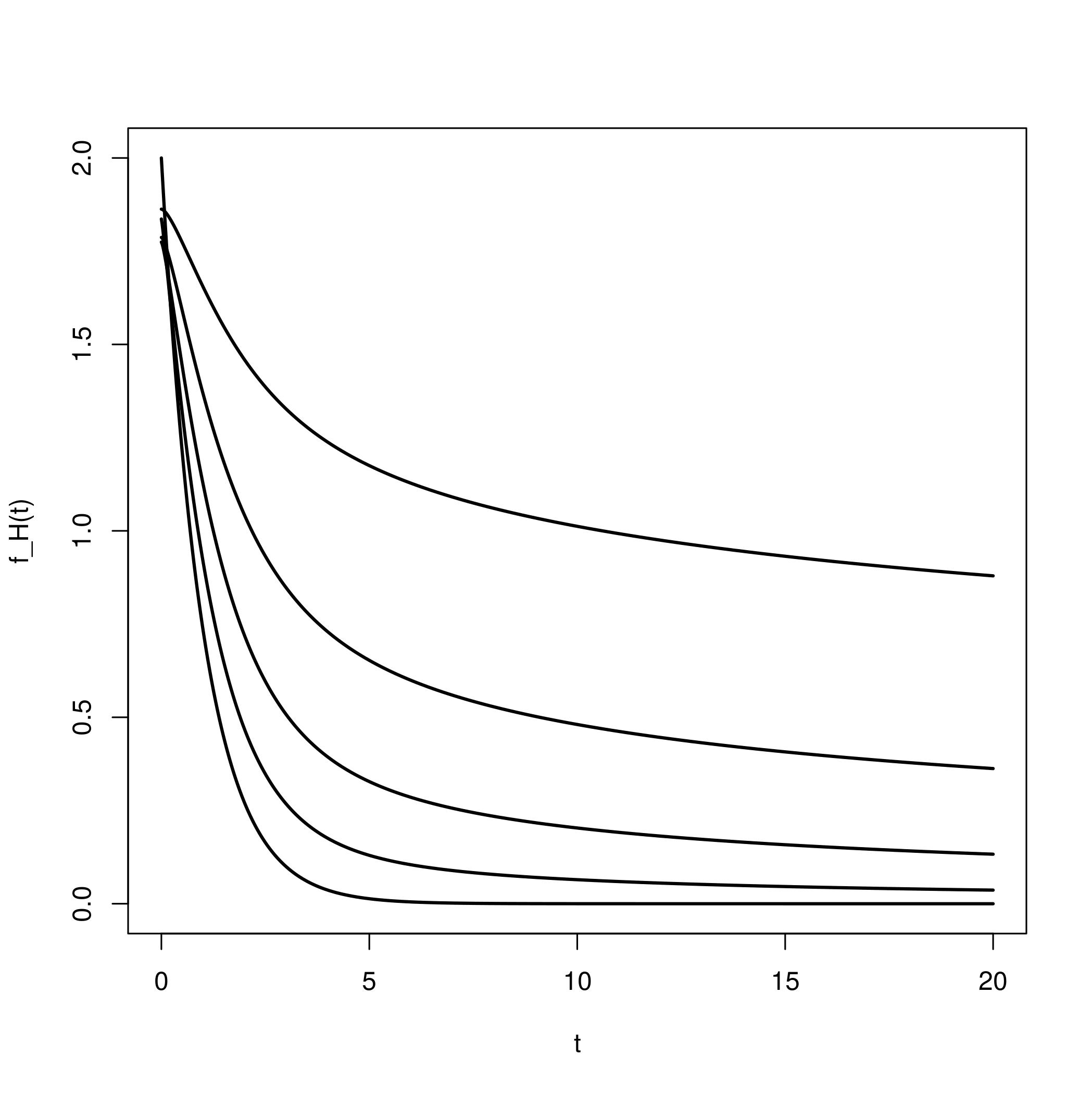} 

   \caption{$f_{H}$ functions for $H=0.5,\ \  0.6, \ \ 0.7, \ \ 0.8$ and$, 0.9$ values. The lowest curve corresponds
   to $H=0.5$ while the highest curve corresponds to $H=0.9$.} 
  \label{f_H}
\end{figure}

Property 1 show us that when $H>1/2$, $f_H(x) \rightarrow 0$  as $x\rightarrow +\infty$, also will be used to prove property 3.
Property 3 show us that when $H>1/2$, then $\sum_{n=1}^{+\infty}f_H(n)=+\infty$. In Figure \ref{f_H} it is perceived the 
slow decrease to zero of function $f_H$ as $H$ increases.
We will prove later in (\ref{covfou}) that the auto-covariance function of any FOU$(\lambda,\sigma,H)$ can be expressed as a multiple 
of $f_H(\lambda t)$, therefore, any FOU$(\lambda,\sigma,H)$ is a long memory process for $H>1/2$.
Property 5 will be used to obtain the spectral density of any FOU$(p)$ process.

The following proposition, is the key that will allows us to express \ the
auto-covariance function of any FOU$\left( \lambda _{1},\lambda
_{2},...,\lambda _{p},\sigma ,H\right) $ as a linear combination of $%
f_{H}\left( \lambda _{i}t\right) .$ The proof it is based on (\ref{pipiras}).

\begin{proposition}
Let $\left\{ X_{t}^{(1)}\right\} _{t\in \mathbb{R}}\sim $FOU$\left( \lambda
_{1},\sigma ,H\right) $\ and \\$\left\{ X_{t}^{(2)}\right\} _{t\in \mathbb{R}%
}\sim $FOU$\left( \lambda _{2},\sigma ,H\right) $\ are generated by the same
fractional Brownian motion $\left\{ \sigma B_{H}(t)\right\} _{t\in \mathbb{R}%
}.$ Then, for all $t$ and $H>1/2$ it is verified that 
\begin{equation}
\mathbb{E}\left( X_{0}^{(1)}X_{t}^{(2)}\right) =\frac{\sigma ^{2}H}{\lambda
_{1}+\lambda _{2}}\left( \lambda _{1}^{1-2H}f_{H}^{\left( 1\right) }(\lambda
_{1}\left\vert t\right\vert )+\lambda _{2}^{1-2H}f_{H}^{\left( 2\right)
}(\lambda _{2}\left\vert t\right\vert )\right) .  \label{cov12}
\end{equation}
\end{proposition}

In particular when $t=0,$\ we get

\begin{equation*}
\mathbb{E}\left( X_{0}^{(1)}X_{0}^{(2)}\right) =\frac{\sigma ^{2}H\Gamma
\left( 2H\right) }{\lambda _{1}+\lambda _{2}}\left( \lambda
_{2}^{1-2H}+\lambda _{1}^{1-2H}\right) =\frac{\sigma ^{2}\Gamma \left(
2H+1\right) }{2\left( \lambda _{1}+\lambda _{2}\right) }\left( \lambda
_{1}^{1-2H}+\lambda _{2}^{1-2H}\right) .
\end{equation*}

If we put $\lambda _{1}=\lambda _{2}=\lambda $ in (\ref{cov12}) we obtain
the auto-covariance function of any FOU$\left( \lambda ,\sigma ,H\right) :$

\begin{corollary}
\bigskip For any $\left\{ X_{t}\right\} _{t\in \mathbb{R}}\sim $FOU$\left(
\lambda ,\sigma ,H\right) $ where $H>1/2$  we get 
\begin{equation}
\mathbb{E}\left( X_{0}X_{t}\right) =\frac{\sigma ^{2}Hf_{H}\left( \lambda
t\right) }{2\lambda ^{2H}}.  \label{covfou}
\end{equation}
\end{corollary}

Observe that property 3 of propositon 8 and (\ref{covfou}), show that any FOU$\left(
\lambda ,\sigma ,H\right) $ is a long memory process.
\begin{remark}
Observe that $f_{H}(0)=2\Gamma \left( 2H\right) ,$ and put $t=0$ in (\ref%
{covfou}), to obtain the known formula for the variance of any FOU$\left(
\lambda ,\sigma ,H\right) :$ 
\begin{equation*}
\mathbb{V}\left( X_{t}\right) =\frac{\sigma ^{2}\Gamma \left( 2H+1\right) }{%
2\lambda ^{2H}}.
\end{equation*}
\end{remark}

In section 4  we prove the following proposition that shows us that the auto-covariance function of any $%
FOU(p)$ where $\lambda _{1},\lambda _{2},...,\lambda _{p}$ are pairwise
different, is a linear combination of the functions $f_{H}\left( \lambda
_{i}t\right) .$

\begin{proposition}
If $\left\{ X_{t}\right\} \sim $FOU$\left( \lambda _{1},\lambda
_{2},...,\lambda _{p},\sigma ,H\right) $ and $p \geq 2$, then 
\begin{equation}
\mathbb{E}\left( X_{0}X_{t}\right) =\frac{\sigma ^{2}H}{2}\sum_{i=1}^{p}%
\frac{\lambda _{i}^{2p-2H-2}}{\prod\limits_{j\neq i}\left( \lambda
_{i}^{2}-\lambda _{j}^{2}\right) }f_{H}(\lambda _{i}t).  \label{covfoup}
\end{equation}
\end{proposition}

\begin{remark}
If $p=1,$ we can consider that  (\ref{covfoup})\bigskip\ it is equal to (\ref{covfou}).
\end{remark}

Using (\ref{covfoup}), and property 5 of the function $f_{H}$ (in Proposition 1) and a little
more work, we obtain Theorem 2, which gives a formula for the spectral density of the
process, that shows that if $p\geq 2,$ then any FOU$\left( p\right) $ is a
short memory \ process.

Observe that when $p=2$, then (\ref{covfoup}) says that \begin{equation}
\mathbb{E}\left( X_{0}X_{t}\right) =\frac{\sigma ^{2}H}{2\left (\lambda_{1}^{2}-\lambda_{2}^{2} \right )}\left ( 
\lambda_{1}^{2-2H}f_{H}(\lambda_{1}t)-\lambda_{2}^{2-2H}f_{H}(\lambda_{2}t)\right ).  \label{covfoup2}
\end{equation}
Now, if we put $\lambda_{1}=\lambda$ and $\lambda_{2} \rightarrow 0$ in (\ref{covfoup2}) we obtain 
(\ref{covfou}). This is FOU$(\lambda_{1},\lambda_{2},
\sigma, H) \rightarrow $ FOU$(\lambda, \sigma, H)$. Therefore, 
for small values of $\lambda_{2}$, the FOU$(\lambda_{1},\lambda_{2},
\sigma, H) $ process can be used to model both short and long memory processes.
\newpage
\begin{theorem}
If $X=\left\{ X_{t}\right\} _{t\in \mathbb{R}}\sim $FOU$\left( \lambda
_{1}^{\left( p_{1}\right) },\lambda _{2}^{\left( p_{2}\right) },...,\lambda
_{q}^{\left( p_{q}\right) },\sigma ,H\right) $ where \\ $%
p_{1}+p_{2}+...+p_{q}=p,$ then the spectral density of the process is
\begin{equation}
f^{\left( X\right) }(x)=\frac{\sigma ^{2}H\Gamma \left( 2H+1\right) \sin
\left( H\pi \right) \left\vert x\right\vert ^{2p-1-2H}}{2\pi
\prod\limits_{i=1}^{q}\left( \lambda _{i}^{2}+x^{2}\right) ^{p_{i}}}.
\label{espectral}
\end{equation}%
In particular, if $\left\{ X_{t}\right\} _{t\in \mathbb{R}}\sim $FOU$\left(
\lambda _{1},\lambda _{2},...,\lambda _{p},\sigma ,H\right) $ then 
\begin{equation}
f^{\left( X\right) }(x)=\frac{\sigma ^{2}H\Gamma \left( 2H+1\right) \sin
\left( H\pi \right) \left\vert x\right\vert ^{2p-1-2H}}{2\pi
\prod\limits_{i=1}^{p}\left( \lambda _{i}^{2}+x^{2}\right) }. \label{espectraldistintos}
\end{equation}
\end{theorem}

\begin{remark}
For $H\in \left( 1/2,1\right) ,$ if $p=1$, then $0$ is a singularity of the
spectral density of the process, then we have a long memory process,
and if $p\geq 2,$ then $0$ is not a singularity of the spectral density of
the process, then we are under a short memory process.
\end{remark}
\section{Applications to real data}

In this section we analize three real data set. In each one of them, we
adjusted different FOU$\left( p\right) $ models \ for $p=2,3,4$, and ARMA
models. To fit the FOU$\left( p\right) $ model, we supose that the real
data set, are indexed in the interval $\left[ 0,T\right] $ for $T=20.$ We
also asume in all of cases that the observations are
equally spaced in time, this is:\ $X_{T/n},X_{2T/n},...,X_{T}.$ A change in the value of $%
T$ results in a change of estimated values of the parameters $\lambda_{i}$ and $\sigma$  but it does not
change the substantial conclusions. That is why we choose arbitrarily $T=20.$

To estimate the parameters of each FOU$\left( p\right) $, we will apply a naive method. We call $\lambda
=\left( \lambda _{1},\lambda _{2},...,\lambda _{q}\right) $  and $\widehat{\gamma }\text{ for the
empirical auto-covariance function,} $ and use 
\begin{equation*}
\left( \widehat{\lambda },\widehat{\sigma },\widehat{H}\right) =\arg \min 
\frac{1}{h}\sum_{i=1}^{h}\left( \gamma \left( iT/n\right) -\widehat{\gamma }%
\left( iT/n\right) \right) ^{2}.
\end{equation*}%
This is, we choose the values of $\left( \lambda ,\sigma ,H\right) $ that
minimize the difference in quadratic mean between the empirical
and theoric auto-covariances in the first $h$ points. The value \ of $h$ was
choosen arbitrarily. In this section, we show the results for $h=10$ terms.
Similar results were obtained for other values of $h.$ The $\arg \min$ 
was taken over the set $\{ (\lambda, \sigma,H) :\ \ \lambda >0, \ \  \sigma >0, \ \  H\geq 1/2 \}.$

In each case, we also fit different ARMA models, and we compare the
performance between these ARMA models and FOU models, through four measures
on the quality of predictions: the root mean square error of prediction for the last $m$ observations,
this is \[RMSE=\sqrt{\frac{1}{m}\sum_{i=1}^{m}\left( X_{n-m+i}-\widehat{X}%
_{n-m+i}\right) ^{2}};\] the mean absolute error of prediction for last $m$
observations and their respectives predictions, this is \[MAE=\frac{1}{m}%
\sum_{i=1}^{m}\left\vert X_{n-m+i}-\widehat{X}_{n-m+i}\right\vert ; \] the
Willmott index (Willmott, 1982) defined by 
\[d=1-\frac{\sum_{i=1}^{m}%
\left( X_{n-m+i}-\widehat{X}_{n-m+i}\right) ^{2}}{\sum_{i=1}^{m}\left(
\left\vert \widehat{X}_{n-m+i}-\overline{X}(m)\right\vert +\left\vert 
\widehat{X}_{n-m+i}-\overline{X}(m)\right\vert \right) ^{2}}\] and Wilmott $L^{1}$ index defined by
\[ d_{1}=1-\frac{\sum_{i=1}^{m}\left\vert
X_{n-m+i}-\widehat{X}_{n-m+i}\right\vert }{\sum_{i=1}^{m}\left( \left\vert 
\widehat{X}_{n-m+i}-\overline{X}(m)\right\vert +\left\vert \widehat{X}%
_{n-m+i}-\overline{X}(m)\right\vert \right) };\]
 where $\overline{X}(m):=%
\frac{1}{m}\sum_{i=1}^{m}X_{n-m+i},$ and $X_{1},X_{2},...,X_{n}$ $\left( 
\text{or }X_{T/n},X_{2T/n},...,X_{T}\right) $ are the real observations, being $%
\widehat{X}_{i}$ are the predictions given by the model for the value $%
X_{i}. $
All the predictions considered are one step.
We will also compare in the three cases the graphs  of empirical auto-covariance function with
those of some fitted models.

\subsection{Oxygen saturation in blood}

The oxygen saturation in blood of a newborn child has been monitored during
seventeen hours, and measures taking every two seconds. We asume that a
series $X_{1},X_{2},...,X_{304}$ of measures taken at intervals of 200
seconds. We adjusted an AR$\left( 1\right) $ and ARMA$\left( 3,3\right) $ to
compare with some FOU models. The AR$\left( 1\right) $ model was choosen
because it maximize Willmott index among all the
ARMA$\left( p,q\right) $ models with $p,q\leq 4.$ The ARMA$\left( 3,3\right) 
$ model was choosen because it was where the maximum AIC value was obtained
between all the ARMA$\left( p,q\right) $ models with $p,q\leq 4.$ 
Figure \ref{acf_oxy} shows \ the empirical auto-covariances  of the series
and the auto-covariances of the adjusted ARMA and FOU models. 
In Figure \ref{oxy2a2} we show the last 
20 observations and their corresponding predictions according to each model. 
We observe that the shape of the predictions for the
FOU$(2)$ model with $\lambda_{1}\neq \lambda_{2}$ is more similar to the observed values than the other models considered.
 In Figure ~\ref{boxplot_oxy} we show the boxplots of  $MAE$ for $%
m=1,2,3,...,20$ predictions for four models adjusted.

In Table 1, we show the values of $d$, $RMSE$, $d_{1}$  and $MAE$, for the last $20$ predictions, 
for different FOU$(p)$ for $p=2,3,4$.
We see that
$\text{FOU}\left( \lambda _{1},\lambda _{2},\sigma ,H\right)$ model achieves a 14\% improvement
over the AR$(1)$ model in Willmott $d$ index, 17 \% in Willmott $d_{1}$ index, 6,5 \% in $MAE$, and 0.6 \% in $RMSE$.
We also observe the good behavior of  $\text{FOU}\left( \lambda ^{\left( 3\right) },\sigma ,H\right)$. Is the best in 
$RMSE$ and $MAE$ and it is very close to the best in $d$ and $d_{1}$. We also see a similar
performance of FOU$(2)$, FOU$(3)$ and FOU$(4)$ for pairwise different  values of $\lambda_{i}$.

\begin{figure}[H]
 \centering
   \includegraphics[scale=0.49]{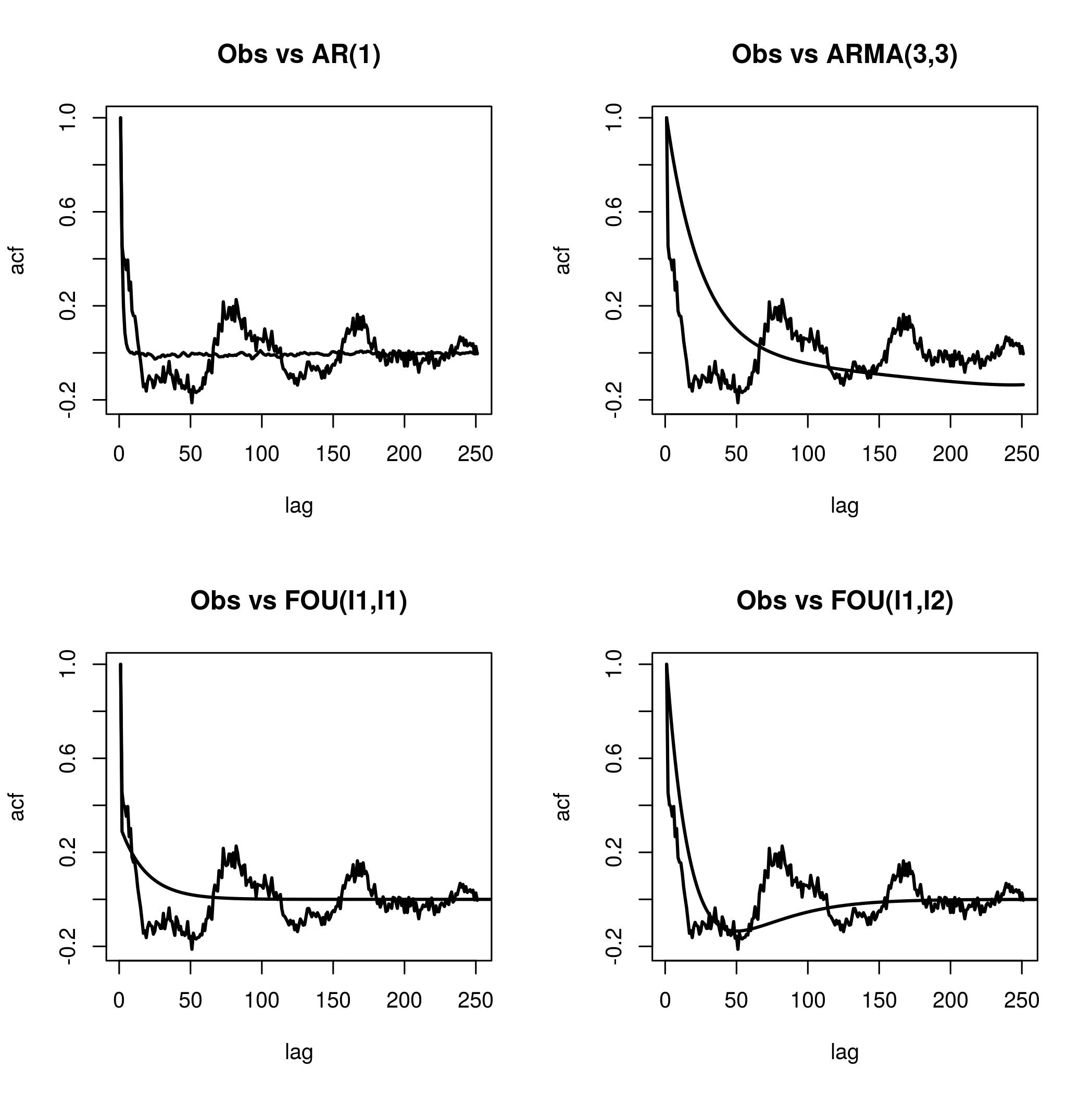} 

   \caption{Empirical auto-covariances vs fitted auto-covariances according to the adjusted model.} 
  \label{acf_oxy}
\end{figure}

\begin{figure}[H]
 \centering
   \includegraphics[scale=0.49]{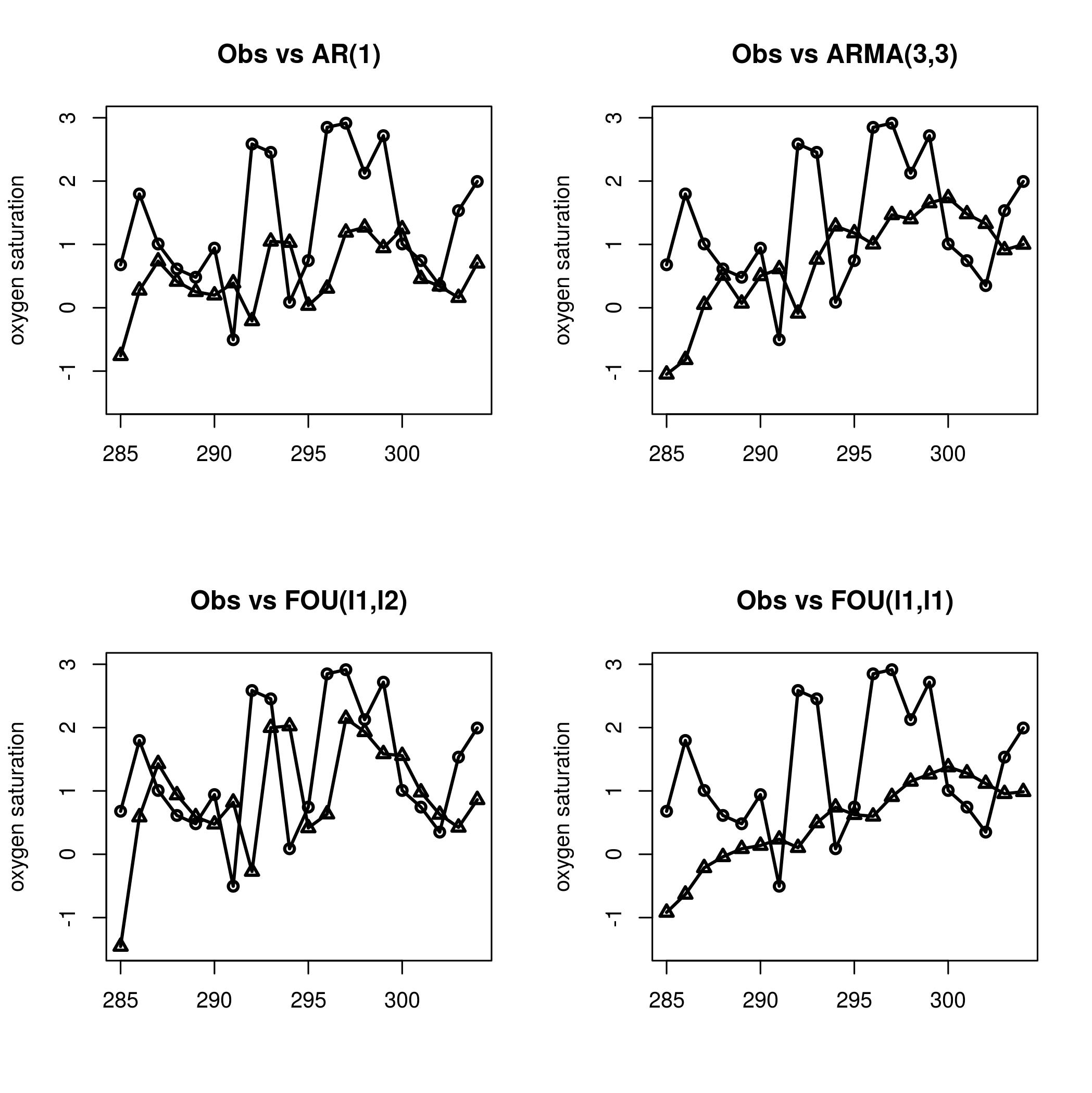} 

   \caption{Last 20 observed values (o) and your corresponding predictions ($ \triangle $) according to te model.} 
  \label{oxy2a2}
\end{figure}

In Figure ~\ref{Willmott_MAE}
we show the \ graph of Willmott index for $m=1,2,3,...,20$ predictions and the graph of the $MAE$. We observe
that for values of $m$ between $6$ to $20$, the  mean absolute predictions error of FOU$(2)$ with $\lambda_{1}\neq\lambda_{2}$ 
values are less than the 
AR$(1)$ and ARMA$(3,3)$ models.
\begin{figure}[H]
 \centering
   \includegraphics[scale=0.382]{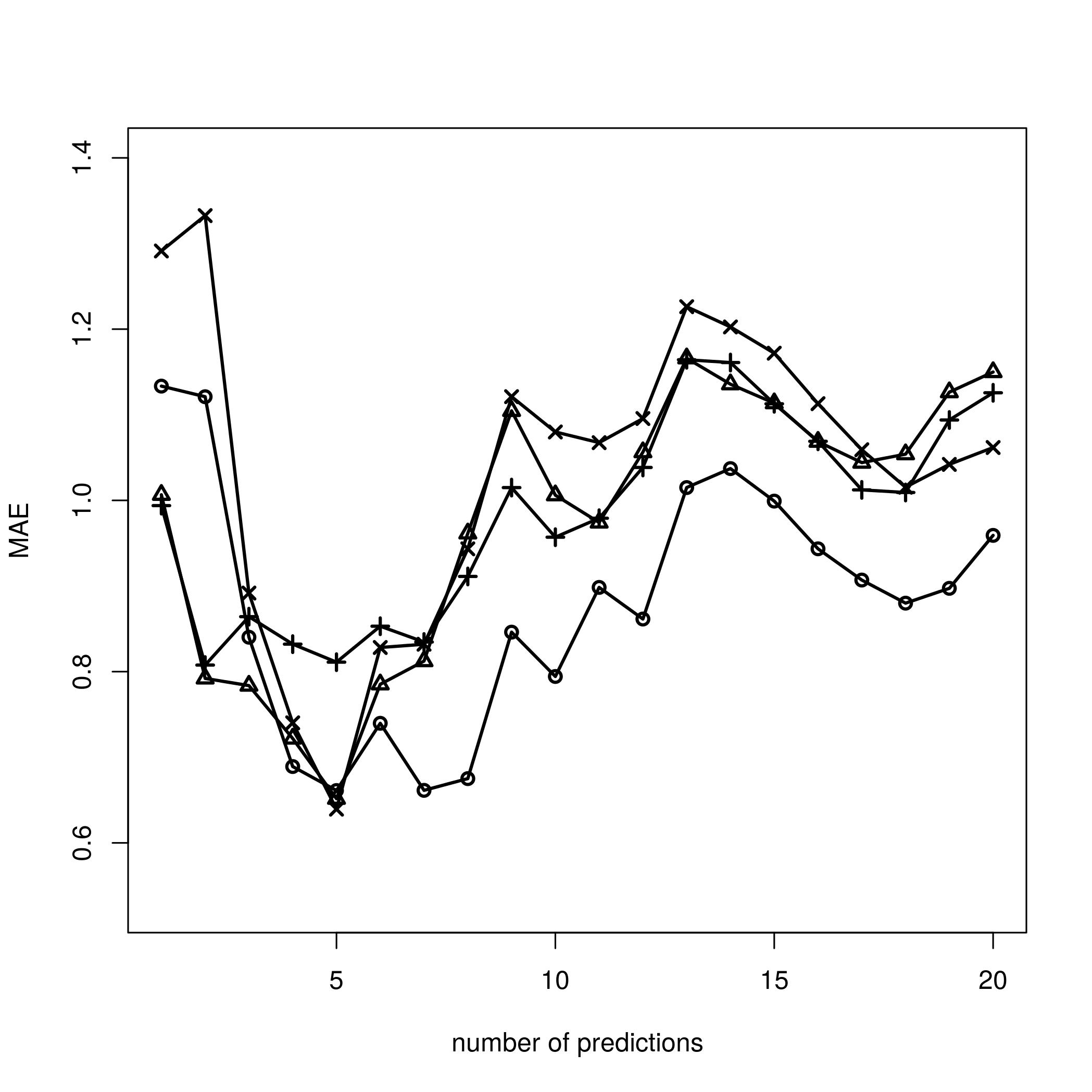} 
   \includegraphics[scale=0.382]{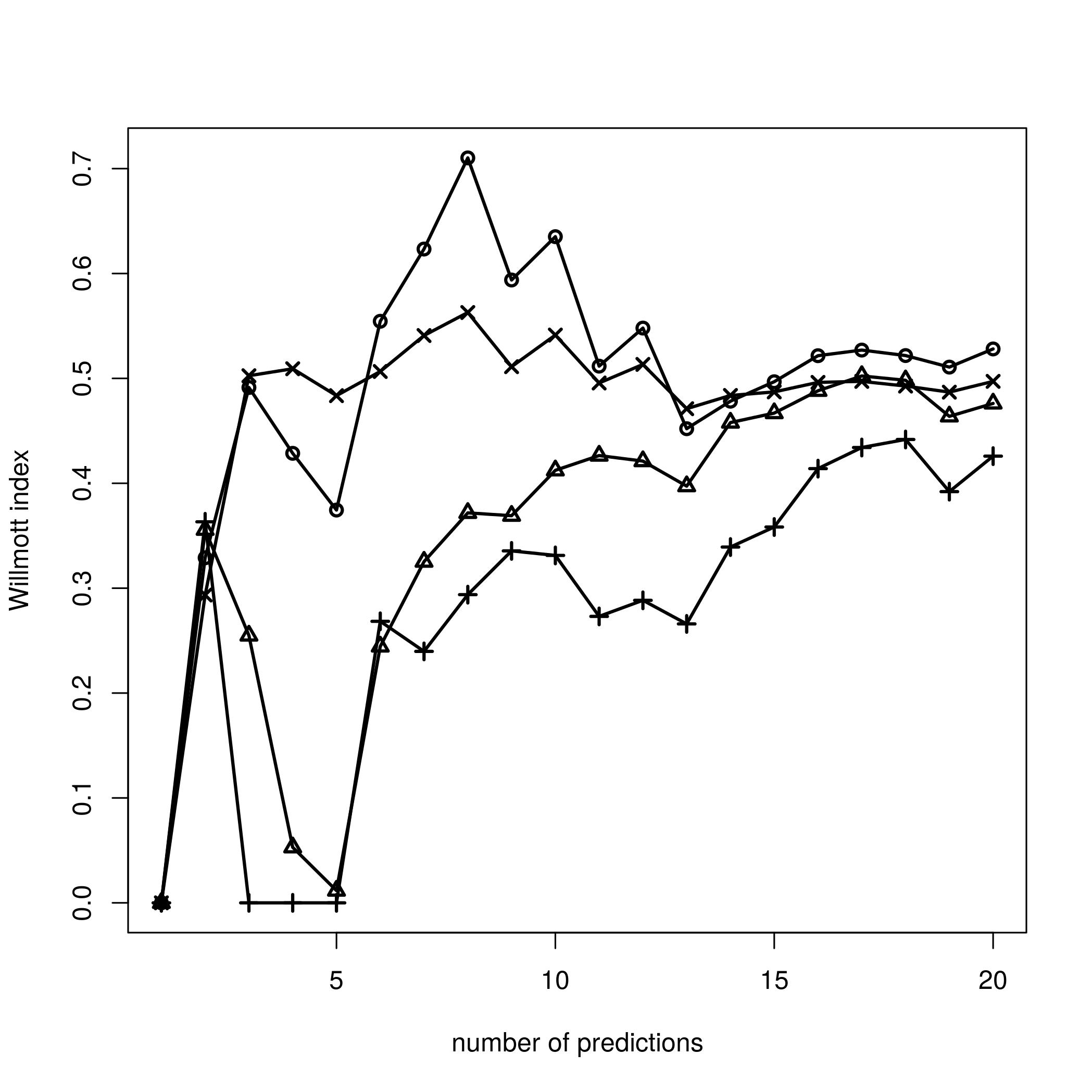} 
   \caption{FOU$(\lambda_{1},\lambda_{2})$ ($\circ$),  FOU$(\lambda^{(2)})$ ($\triangle$),  ARMA$(3,3)$ ($ +$), 
    AR$(1)$ ($\times$).  MAE for the last $m=1,2,3,...,20$ observations for the four models, and Willmott values. 
    } 
  \label{Willmott_MAE}
\end{figure}

\begin{figure}[H]
 \centering
   \includegraphics[scale=0.44]{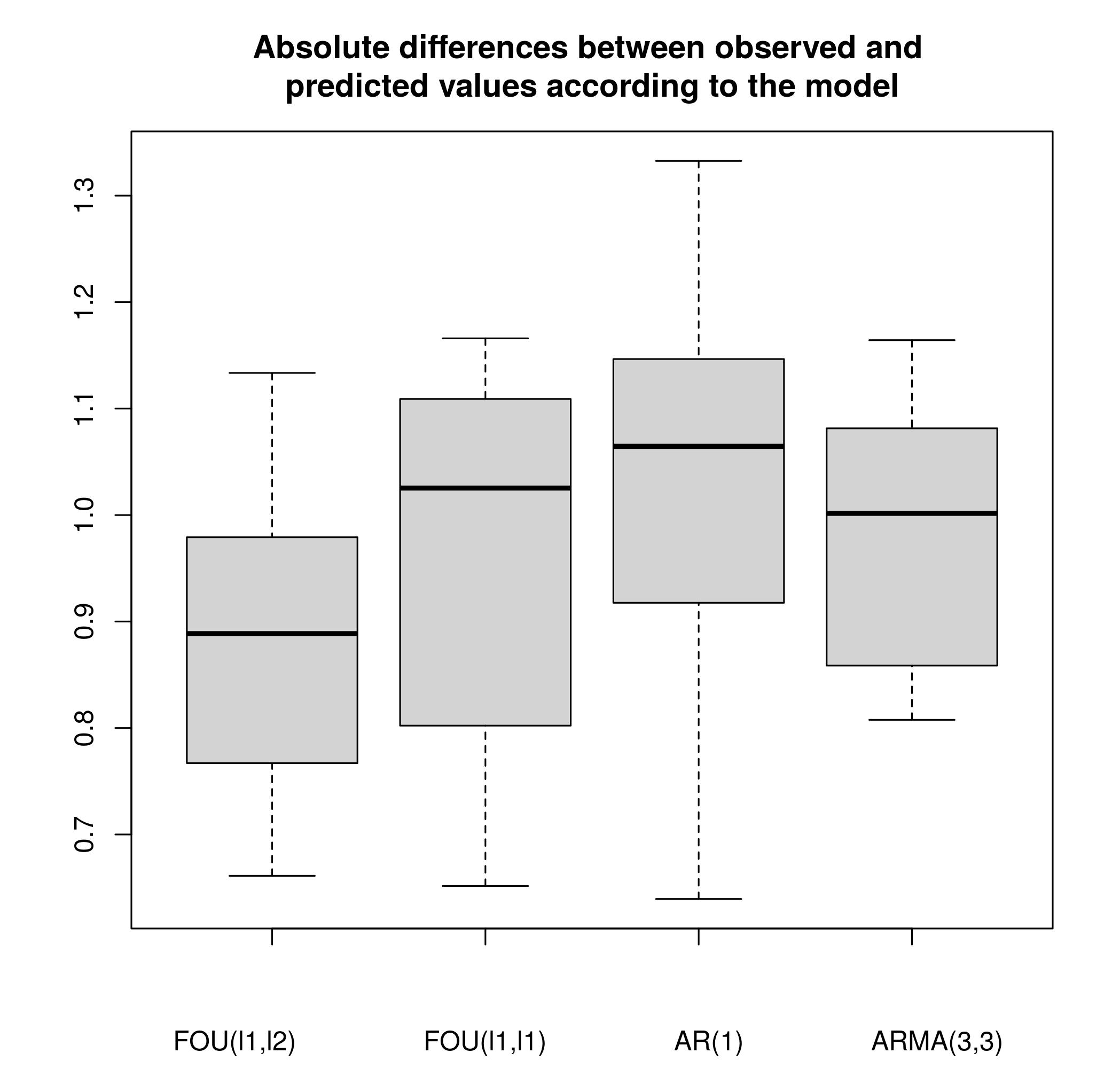} 

   \caption{Absolute differences between observed and predicted values according to the model.} 
  \label{boxplot_oxy}
\end{figure}

\begin{center}

Table 1. Values of $d$, $RMSE$, $d_{1}$  and $MAE$ for different models. 

\fbox{$%
\begin{array}{c|cccc}
\text{Model} & d & RMSE & d_{1} & MAE \\ 
\hline
\text{AR}\left( 1\right) & 0.4972 & 1.3051 & 0.3866 & 1.0617 \\ 
\text{ARMA}\left( 3,3\right) & 0.4259 & 1.3148 & 0.2868 & 1.1257 \\ 
\text{FOU}\left( \lambda _{1},\lambda _{2},\sigma ,H\right) & 0.5663 & 1.2967
& 0.4513 & 0.9967 \\ 
\text{FOU}\left( \lambda _{1},\lambda _{2},\lambda _{3},\sigma ,H\right) & 
0.5653 & 1.2803 & 0.4482 & 0.9894 \\ 
\text{FOU}\left( \lambda _{1},\lambda _{2},\lambda _{3},\lambda _{4},\sigma
,H\right) & 0.5643 & 1.2783 & 0.4476 & 0.9882 \\ 
\text{FOU}\left( \lambda ^{\left( 2\right) },\sigma ,H\right) & 0.4754 & 
1.3565 & 0.3303 & 1.1554 \\ 
\text{FOU}\left( \lambda ^{\left( 3\right) },\sigma ,H\right) & 0.5652 & 
1.2750 & 0.4479 & 0.9858 \\ 
\text{FOU}\left( \lambda ^{\left( 4\right) },\sigma ,H\right) & 0.5607 & 
1.2980 & 0.4474 & 1.0015%
\end{array}%
$}
\end{center}

\subsection{Box, Jenkins and Reinsel ``series A"}

The Series A is a record of $197$ chemical process concentration readings,
taken every two hours. This series was introduced by Box et al (Box et al,
1994, Ch. 4), also suggest an ARMA$\left( 1,1\right) $ to this data set. An
AR$\left( 7\right) $ are proposed in (Cleveland, 1971) and (McLeod \& Zang,
2006). In Figure \ref{acfseriesA} we observe that auto-covariances of AR$(7)$ and ARMA$(1,1)$ adjusted models, goes to zero very quickly
and their auto-covariance structure does not resemble that observed.

\begin{figure}[H]
 \centering
   \includegraphics[scale=0.49]{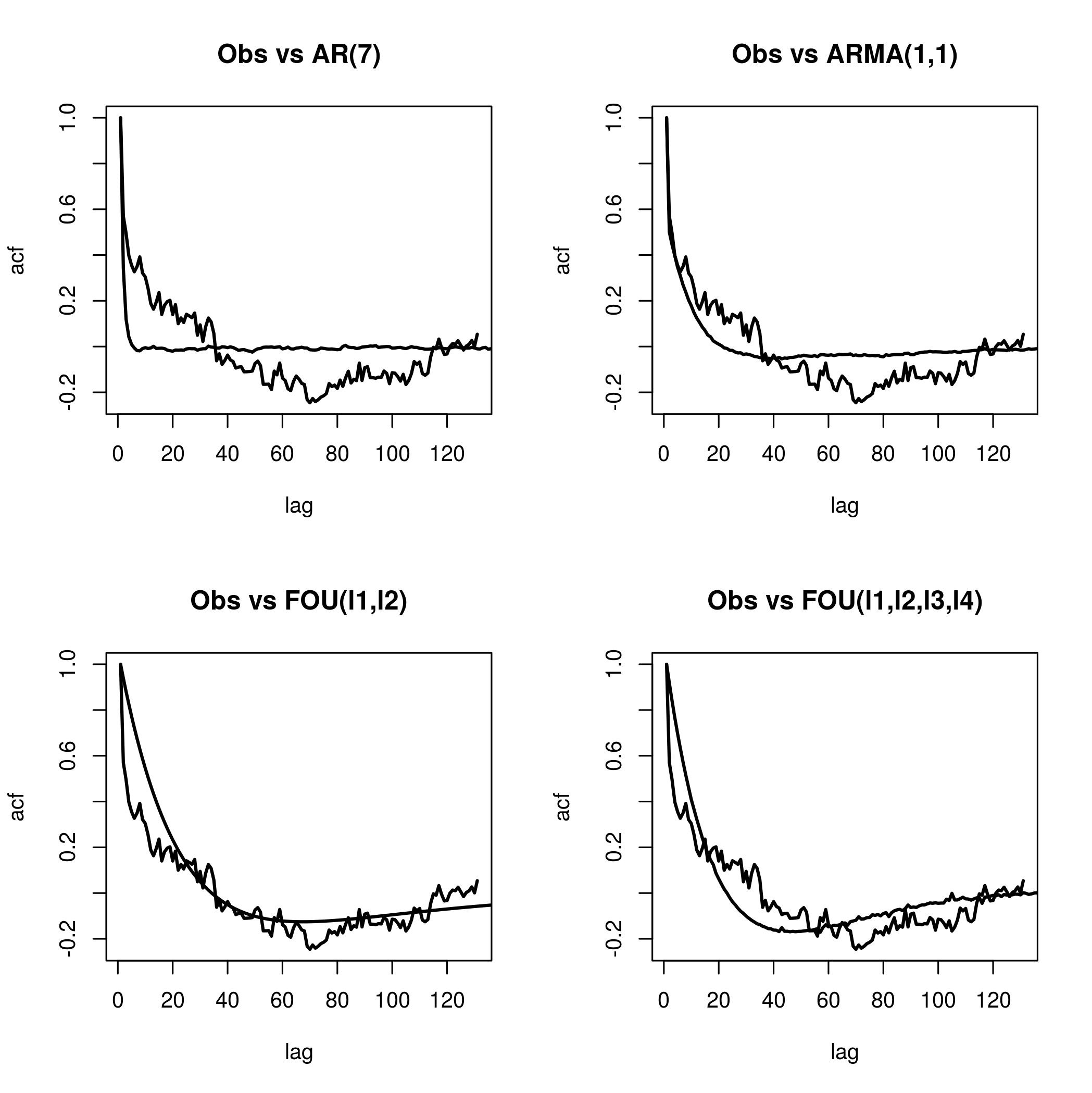} 
   \caption{Empirical auto-covariances vs fitted auto-covariances according to the adjusted model for 
   series A data set.} 
  \label{acfseriesA}
\end{figure}
In Figure \ref{seriesA_obs_2a2} we show the last 
20 observations and their corresponding predictions according to each model. 
\begin{figure}[H]
 \centering
   \includegraphics[scale=0.49]{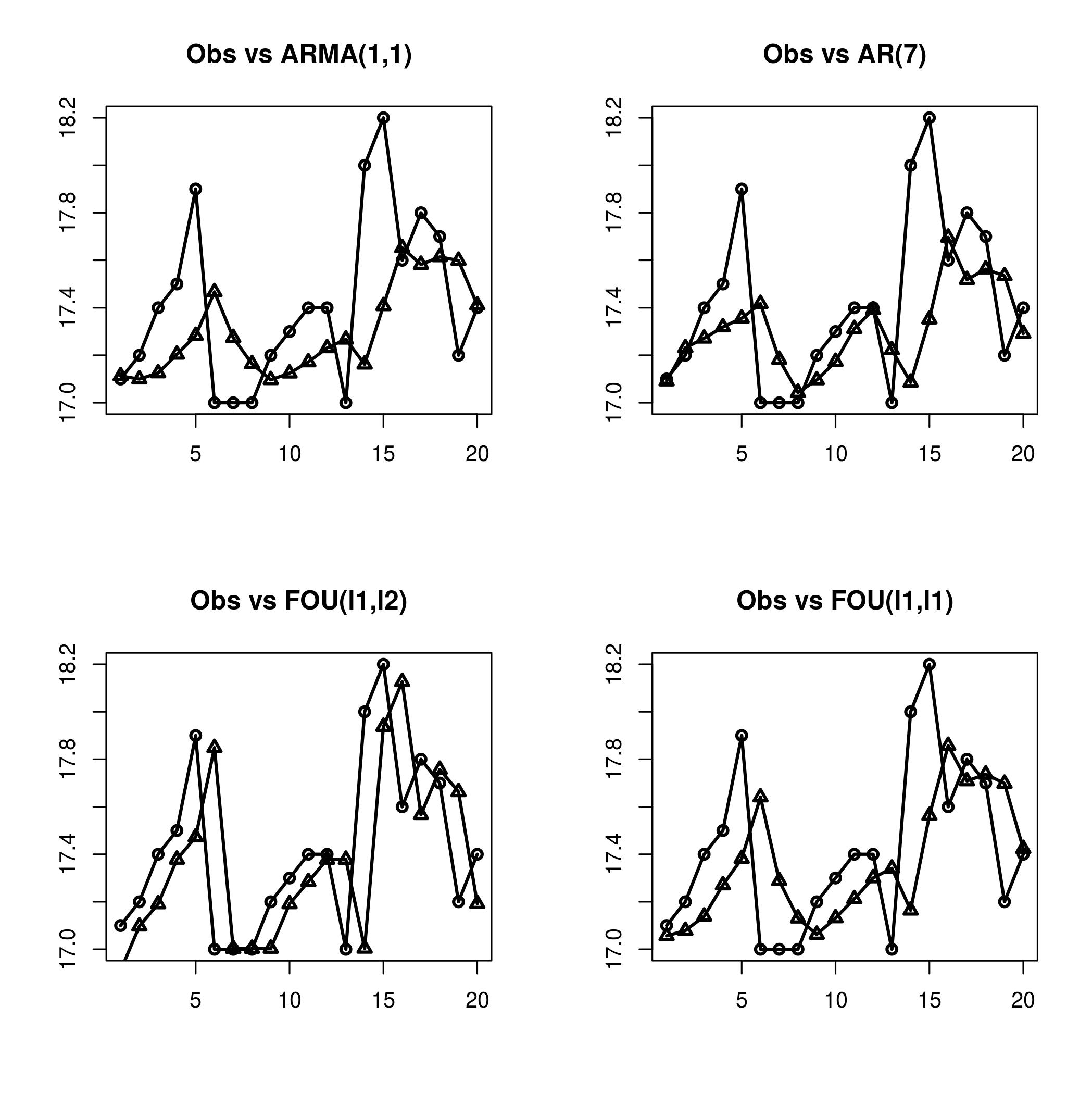} 
   \caption{Last 20 observed values (o) and your corresponding predictions ($ \triangle $) according to te model for 
   series A data set.} 
  \label{seriesA_obs_2a2}
\end{figure}

\begin{figure}[H]
 \centering
   \includegraphics[scale=0.382]{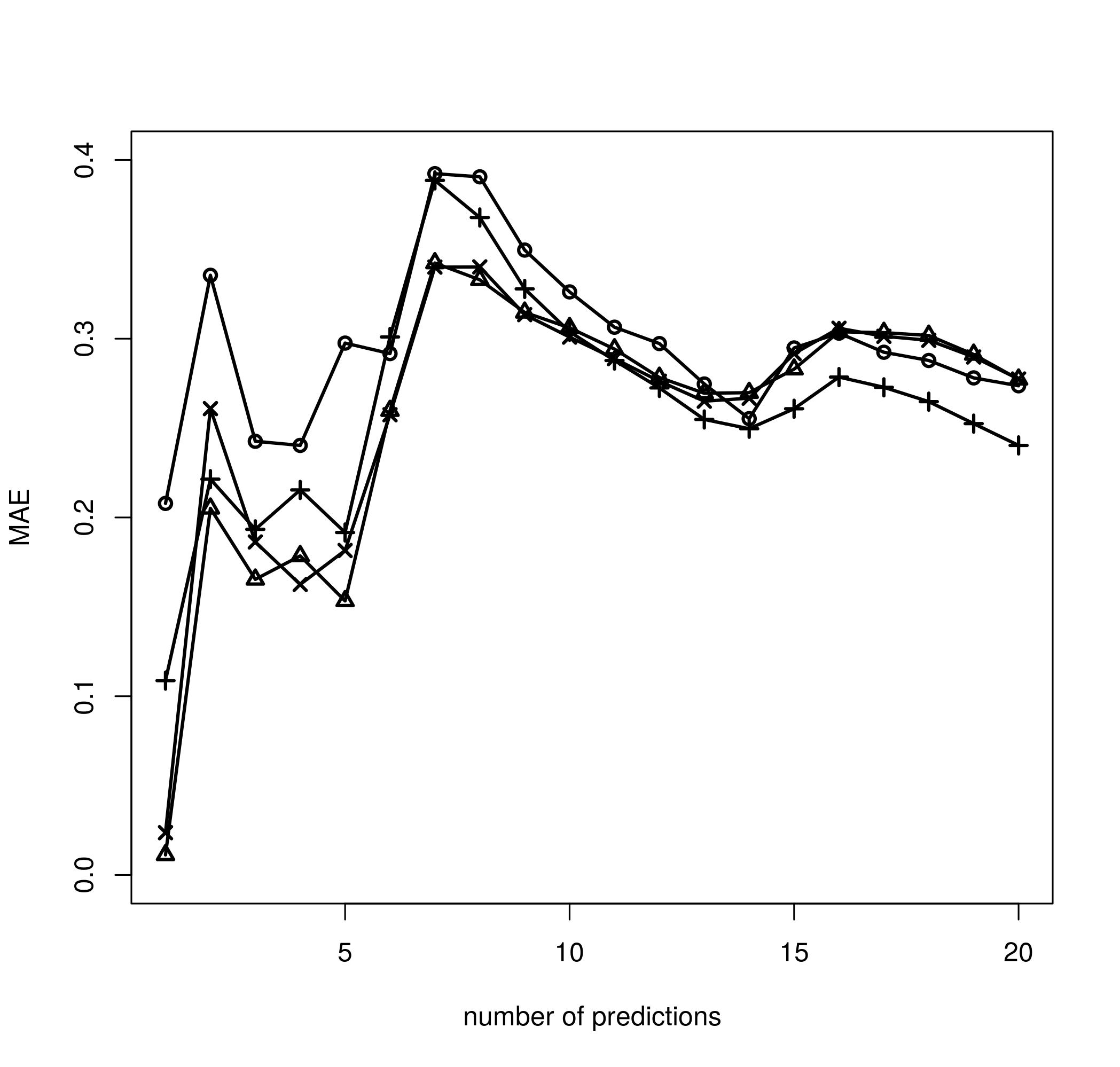} 
   \includegraphics[scale=0.382]{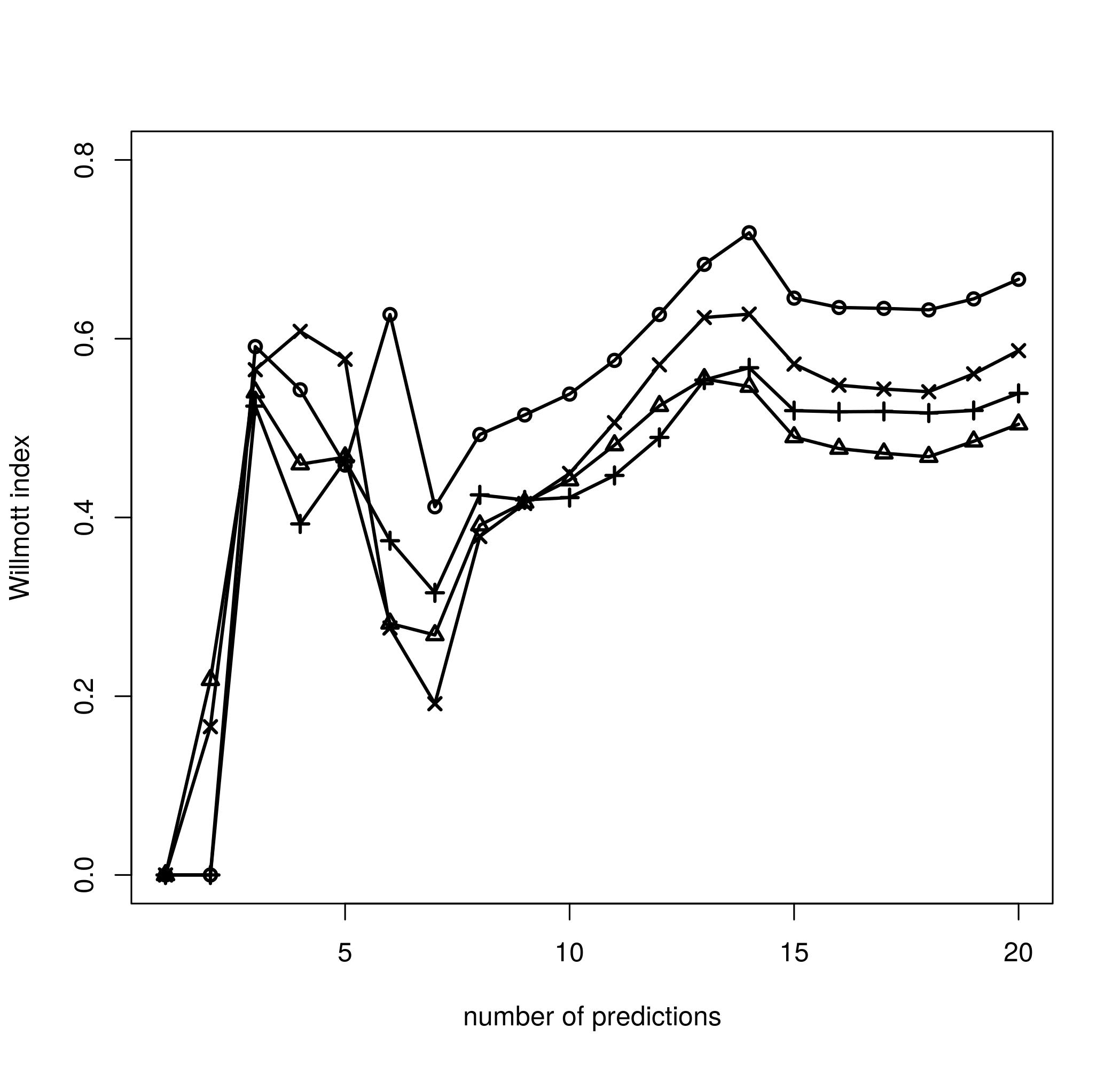} 
   \caption{FOU$(\lambda_{1},\lambda_{2})$ ($\circ$),  FOU$(\lambda^{(2)})$ ($\times$),  ARMA$(1,1)$ ($ \triangle$), 
    AR$(7)$ ($+$). MAE for the last $m=1,2,3,...,20$ observations for the four models, and Willmott values in 
   seriesA.} 
  \label{locon}
\end{figure}

 Again, like the oxygen saturation in blood data set, 
we see that the graph of predictions generated by FOU$(2)$ with $\lambda_{1}\neq\lambda_{2}$ model, have a shape more similar
to the observed curve  than those generated by the other models.
In Table 2, we show the values of $d$, $RMSE$, $d_{1}$  and $MAE$ for AR$(7)$, ARMA$(1,1)$ and different FOU$(p)$ for $p=2,3,4$ 
models. 
We see that
$\text{FOU}\left( \lambda _{1},\lambda _{2},\sigma ,H\right)$ model achieves a 23\% improvement
over the AR$(7)$ model in Willmott $d$ index, 11 \% in Willmott $d_{1}$ index, but has a loss of 14 \% in $MAE$, and 8 \% in $RMSE$.
We also observe the good behavior of  $\text{FOU}\left( \lambda ^{\left( 4\right) },\sigma ,H\right)$. The models
 FOU$(2)$, FOU$(3)$ and FOU$(4)$ for pairwise different  values of $\lambda_{i}$, are performing worse as the number of
 parameters increases in terms of $MAE$ and $RMSE$.
 
\noindent 
\begin{center}
 
Table 2. Values of $d$, $RMSE$, $d_{1}$  and $MAE$ for different models adjusted to Series A.

\fbox{$%
\begin{array}{c|cccc}
\text{Model} & d & RMSE & d_{1} & MAE \\ 
\hline
\text{AR}\left( 7\right) & 0.5389 & 0.3482 & 0.4690 & 0.2403 \\ 
\text{ARMA}\left( 1,1\right) & 0.5043 & 0.3615 & 0.4175 & 0.2773 \\ 
\text{FOU}\left( \lambda _{1},\lambda _{2},\sigma ,H\right) & 0.6665 & 0.3788
& 0.5215 & 0.2737 \\ 
\text{FOU}\left( \lambda _{1},\lambda _{2},\lambda _{3},\sigma ,H\right) & 
0.6452 & 0.4462 & 0.4940 & 0.3204 \\ 
\text{FOU}\left( \lambda _{1},\lambda _{2},\lambda _{3},\lambda _{4},\sigma
,H\right) & 0.6509 & 0.4763 & 0.4839 & 0.3439 \\ 
\text{FOU}\left( \lambda ^{\left( 2\right) },\sigma ,H\right) & 0.5866 & 
0.3574 & 0.4534 & 0.2774 \\ 
\text{FOU}\left( \lambda ^{\left( 3\right) },\sigma ,H\right) & 0.6401 & 
0.4339 & 0.4895 & 0.3141 \\ 
\text{FOU}\left( \lambda ^{\left( 4\right) },\sigma ,H\right) & 0.6662 & 
0.3796 & 0.5313 & 0.2743%
\end{array}%
$}
\end{center}

\subsection{Level in feet of Lake Huron}

Level in feet of years 1875 to 1972, is a time series of $98$ observations.

\begin{figure}[H]
 \centering
   \includegraphics[scale=0.49]{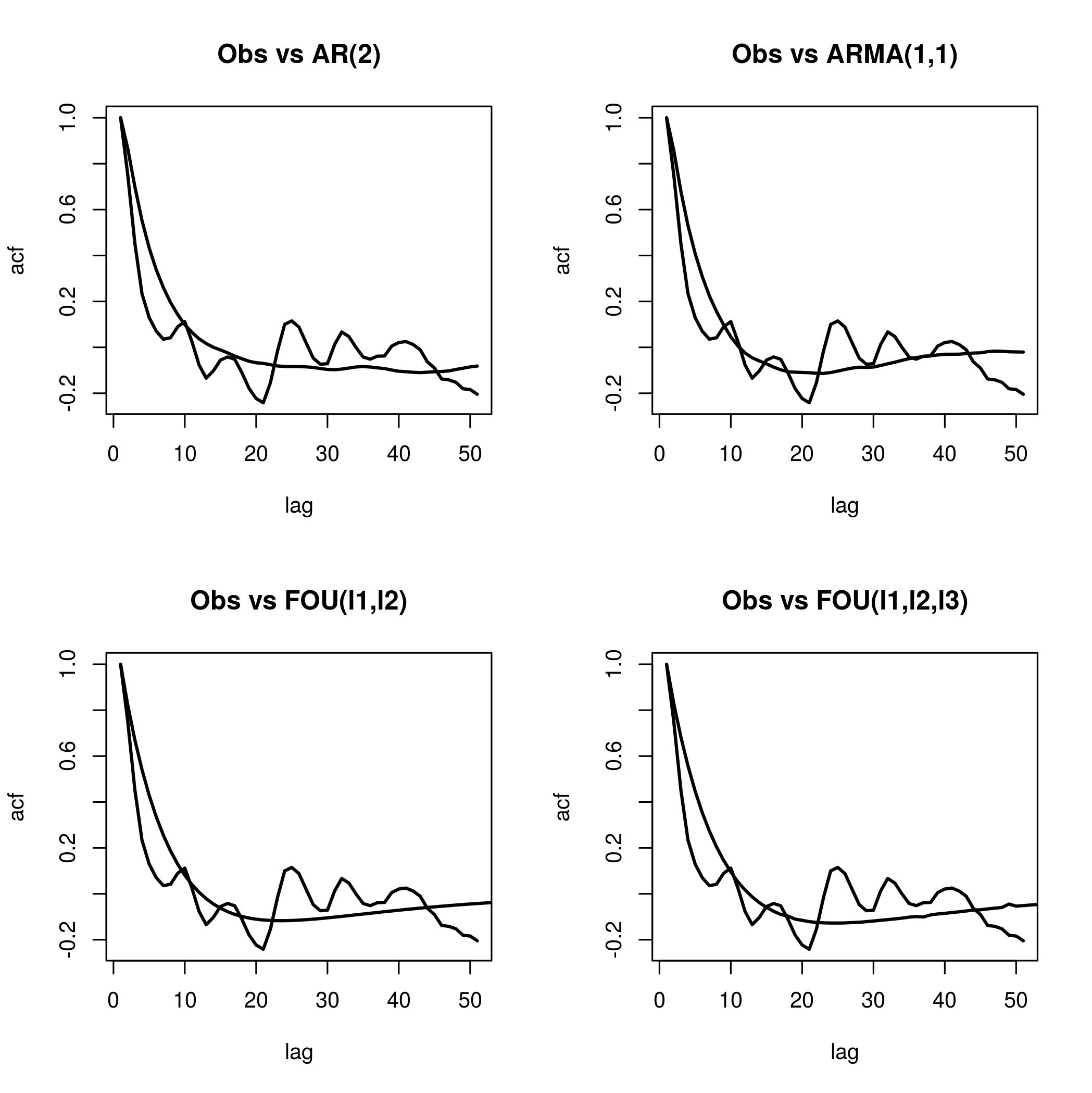} 
   \caption{Empirical auto-covariances vs fitted auto-covariances according to the adjusted model for Lake Huron data set.} 
  \label{acfHuron}
\end{figure}

\begin{figure}[H]
 \centering
   \includegraphics[scale=0.49]{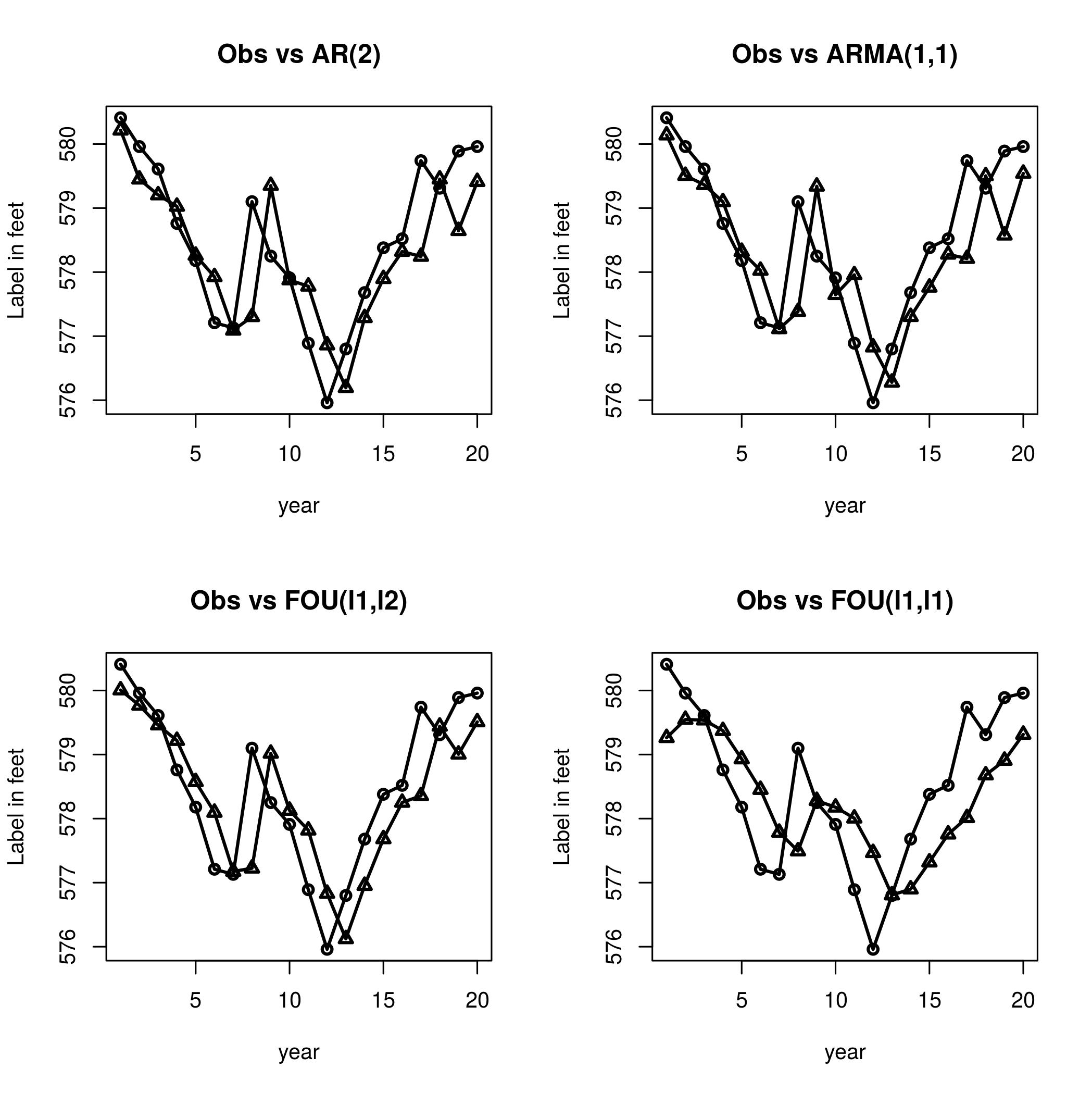} 
   \caption{Last 20 observed values (o) and your corresponding predictions ($ \triangle $) according to te model for Lake Huron
   data set.} 
  \label{Huronobs2a2}
\end{figure}

\begin{figure}[H]
 \centering
   \includegraphics[scale=0.382]{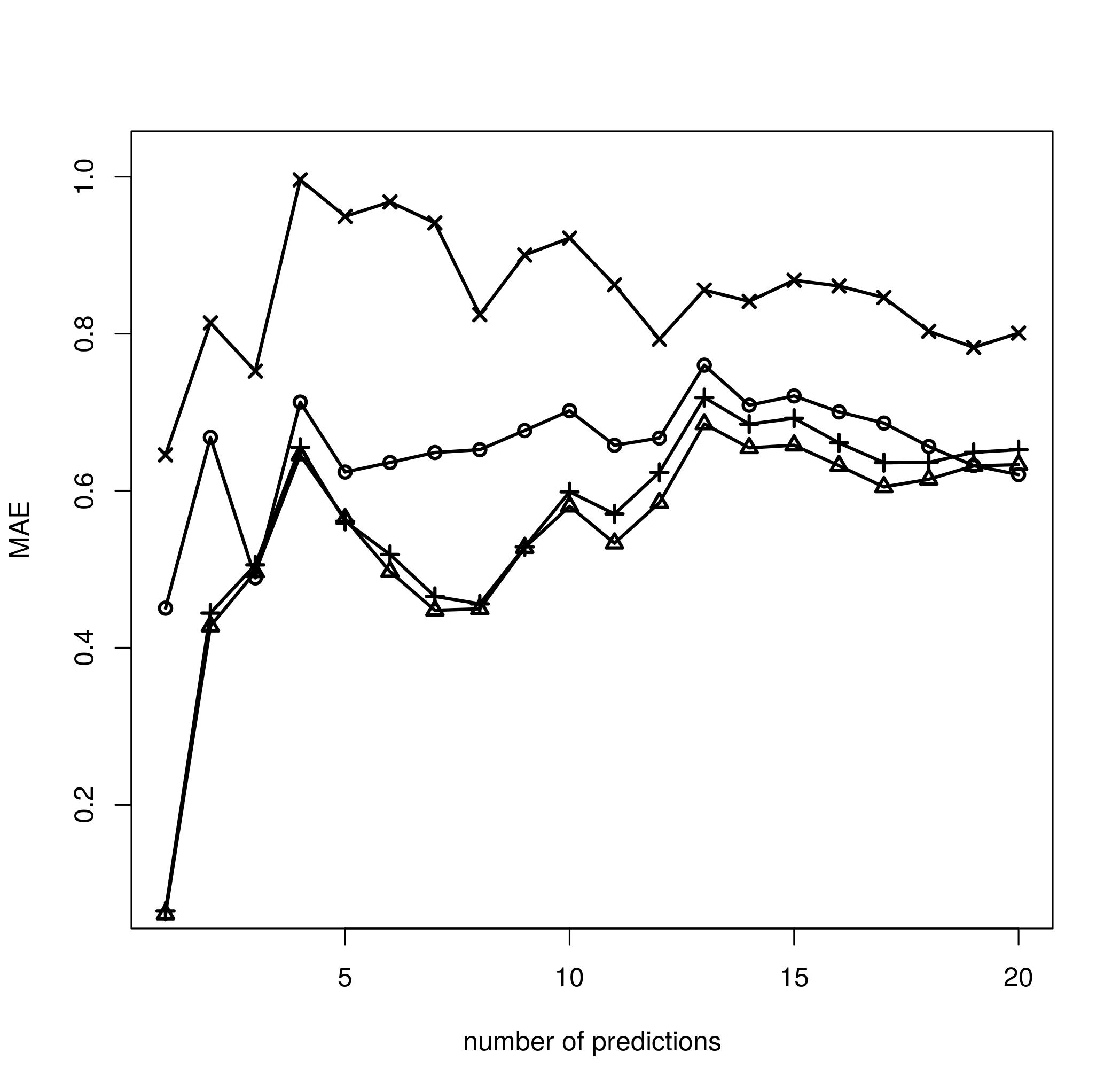} 
   \includegraphics[scale=0.382]{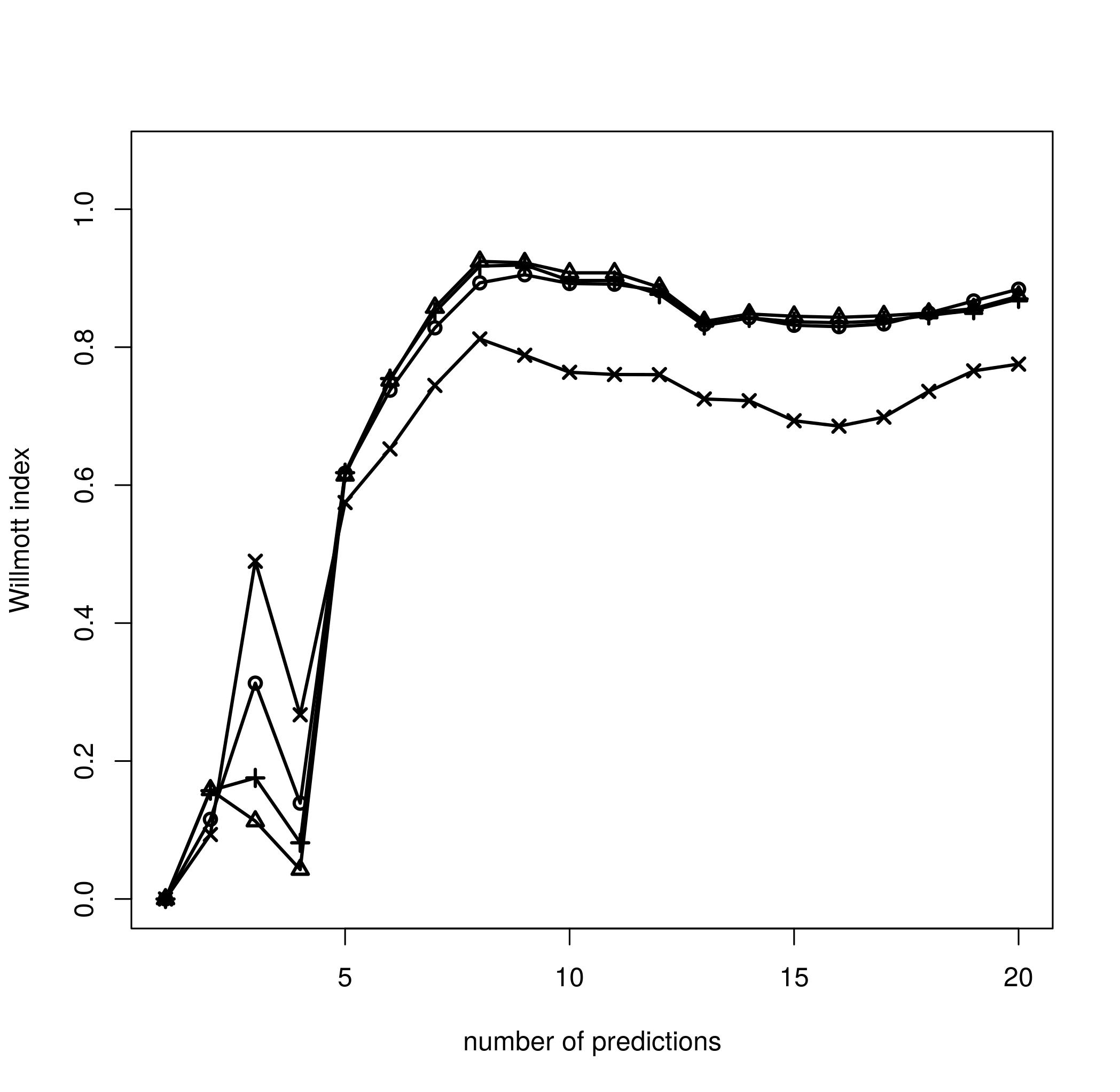} 
\caption{FOU$(\lambda_{1},\lambda_{2})$ ($\circ$),  FOU$(\lambda^{(2)})$ ($\times$),  ARMA$(1,1)$ ($ +$), 
    AR$(2)$ ($\triangle$). MAE for the last $m=1,2,3,...,20$ observations for the four models, and Willmott values for 
   Lake Huron data set.} 
  \label{locon}
\end{figure}
The series has a slight tendency that was removed before
adjusting the models. In (Brockwell \& Davis, 2002), suggest an AR$\left(
2\right) $ and ARMA$\left( 1,1\right) $ to this series. 
In Figure \ref{acfHuron} we observe the auto-covariances of AR$(2)$, ARMA$(1,1)$ and two FOU$(2)$ adjusted  models. In 
this case there are no substantial differences between the different models adjusted.
Nor are there significant differences between the observed curve and the predictions curve for the the different models in the last 20 
observations (Figure \ref{Huronobs2a2}).
\begin{figure}[H]
 \centering
   \includegraphics[scale=0.45]{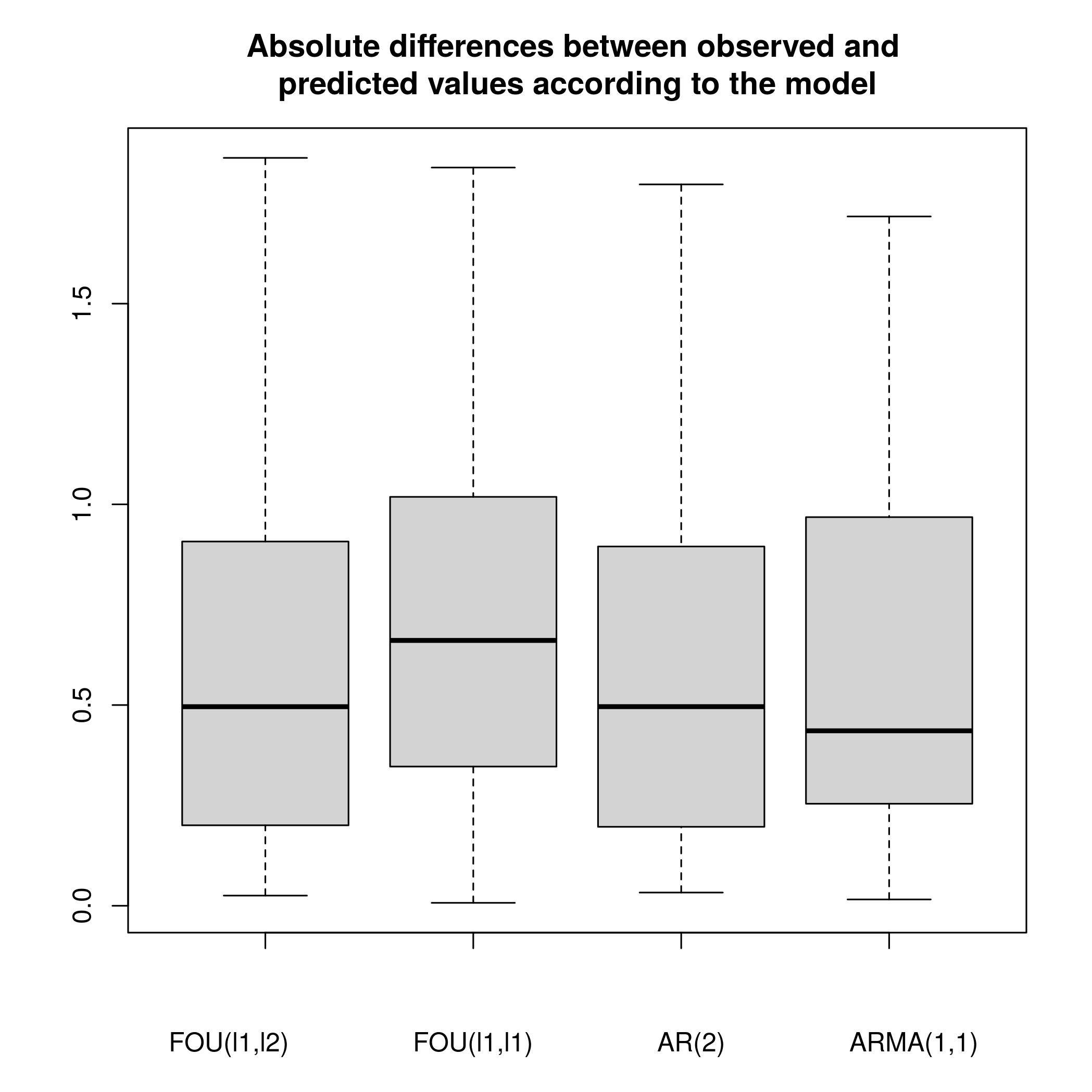} 
 \caption{MAE values according to the model.}  
  \label{loconnn}
\end{figure}

In Table 3, we show the values of $d$, $RMSE$, $d_{1}$  and $MAE$, for adjusted AR$(2)$, ARMA$(1,1)$ and different FOU$(p)$ for $p=2,3,4$
models.

We see that the performance of all models considered are similar (except $\text{FOU}\left( \lambda^{(3)},\sigma ,H\right)$ 
and FOU$(\lambda^{(2)}, \sigma, H)$ models).
Anyway, we see that 
$\text{FOU}\left( \lambda _{1},\lambda _{2},\lambda_{3},\lambda_{4},\sigma ,H\right)$ model shows a slightly better results.

\begin{center}
 Table 3. Values of  $d$, $RMSE$, $d_{1}$  and $MAE$ for different models, adjusted to the series ``level on feet, Lake Huron''.

\fbox{$%
\begin{array}{c|cccc}
\text{Model} & d & RMSE & d_{1} & MAE \\ 
\hline
\text{AR}\left( 2\right) & 0.8739 & 0.7891 & 0.6735 & 0.6331 \\ 
\text{ARMA}\left( 1,1\right) & 0.8700 & 0.7994 & 0.66655 & 0.6523 \\ 
\text{FOU}\left( \lambda _{1},\lambda _{2},\sigma ,H\right) & 0.8841 & 0.7620
& 0.6833 & 0.6205 \\ 
\text{FOU}\left( \lambda _{1},\lambda _{2},\lambda _{3},\sigma ,H\right) & 
0.8872 & 0.7998 & 0.7023 & 0.6169 \\ 
\text{FOU}\left( \lambda _{1},\lambda _{2},\lambda _{3},\lambda _{4},\sigma
,H\right) & 0.8963 & 0.7837 & 0.7197 & 0.5938 \\ 
\text{FOU}\left( \lambda ^{\left( 2\right) },\sigma ,H\right) & 0.7755 & 
0.9394 & 0.5538 & 0.8007 \\ 
\text{FOU}\left( \lambda ^{\left( 3\right) },\sigma ,H\right) & 0.5915 & 
2.1090 & 0.5818 & 1.1678 \\ 
\text{FOU}\left( \lambda ^{\left( 4\right) },\sigma ,H\right) & 0.8752 & 
0.7755 & 0.6731 & 0.6294%
\end{array}%
$}
\end{center}

We see that FOU$(\lambda_{1},\lambda_{2},\lambda_{3},\lambda_{4})$ achieves the best results in $d$, $d_{1}$ and $MAE$, 
but the results for FOU$(\lambda^{(4)})$ are very similars and has fewer parameters. In $RSME$ the best result is obtained in 
FOU$(\lambda^{(4)})$ model.
\section{Proofs}

\begin{proof}[Proof of Proposition 1]
\begin{enumerate}
\item It is enough to prove that 
\begin{equation*}
\frac{-\alpha ^{1-2H}\int_{0}^{\alpha x}e^{s}s^{2H-1}ds+\beta
^{1-2H}e^{\left( \alpha +\beta \right) x}\int_{\beta x}^{+\infty
}e^{-s}s^{2H-1}ds}{e^{\alpha x}}\rightarrow 0.
\end{equation*}

We apply L'H\^{o}pital rule two times and we obtain that 
\begin{equation*}
\underset{x\rightarrow +\infty }{\text{lim}}\frac{\alpha +\beta }{\alpha }%
\frac{-e^{-\beta x}x^{2H-1}+\beta ^{1-2H}\int_{\beta x}^{+\infty
}e^{-s}s^{2H-1}ds}{e^{-\beta x}}=
\end{equation*}%
\begin{equation*}
\underset{x\rightarrow +\infty }{\text{lim}}\frac{\alpha +\beta }{\alpha
\beta }\left( 2H-1\right) x^{2H-2}\rightarrow 0.
\end{equation*}

\item Due $\alpha ^{1-2H}f_{H}^{\left( 1\right) }(\alpha x)+\beta
^{1-2H}f_{H}^{\left( 2\right) }(\beta x)\rightarrow 0$ and observe that $%
e^{-x}/x^{2-2H}\rightarrow 0$ as $x\rightarrow +\infty ,$ we apply L'H\^{o}%
pital rule and get 
\begin{equation*}
\underset{x\rightarrow +\infty }{\text{lim}}\frac{\alpha
^{1-2H}f_{H}^{\left( 1\right) }(\alpha x)+\beta ^{1-2H}f_{H}^{\left(
2\right) }(\beta x)}{x^{2H-2}}=
\end{equation*}%
\begin{equation*}
\underset{x\rightarrow +\infty }{\text{lim}}\frac{\beta ^{1-2H}e^{\left(
\alpha +\beta \right) x}\int_{\beta x}^{+\infty }e^{-s}s^{2H-1}ds-\alpha
^{1-2H}\int_{0}^{\alpha x}e^{s}s^{2H-1}ds}{e^{\alpha x}x^{2H-2}}=
\end{equation*}

\begin{equation*}
\underset{x\rightarrow +\infty }{\text{lim}}\frac{\alpha +\beta }{\alpha }%
\frac{-x^{2H-1}+\beta ^{1-2H}e^{\beta x}\int_{\beta x}^{+\infty
}e^{-s}s^{2H-1}ds}{x^{2H-2}}=\frac{\alpha +\beta }{\alpha \beta }\left(
2H-1\right)
\end{equation*}%
where in the last equality was applied again the L'H\^{o}pital rule.

\item In property 2, put $\alpha =\beta =\lambda $, and we get that 
\begin{equation*}
f_{H}(\lambda x)=f_{H}^{\left( 1\right) }(\lambda x)+f_{H}^{\left( 2\right)
}(\lambda x)\sim 2\left( 2H-1\right) \left( \lambda x\right) ^{2H-2}
\end{equation*}%
where $x\rightarrow +\infty .$

\item 
\begin{equation*}
f_{H}(x)-f_{H}(0)=f_{H}(x)-2\Gamma \left( 2H\right) =
\end{equation*}%
\begin{equation*}
\Gamma \left( 2H\right) \left( e^{x}+e^{-x}-2\right)
-e^{-x}\int_{0}^{x}e^{s}s^{2H-1}ds-e^{x}\int_{0}^{x}e^{-s}s^{2H-1}ds=
\end{equation*}%
\begin{equation*}
o\left( x^{2H}\right) -e^{x}\sum_{n=0}^{+\infty }\frac{x^{n+2H}}{n!(n+2H)}%
-e^{-x}\sum_{n=0}^{+\infty }\frac{\left( -1\right) ^{n}x^{n+2H}}{n!(n+2H)}%
=o\left( x^{2H}\right) -\frac{x^{2H}}{H}.
\end{equation*}

\item (Pipiras \& Taqqu, 2000), shown that if $X_{t}\sim $%
FOU$\left( \lambda ,\sigma ,H\right) ,$ then 
\begin{equation*}
\rho \left( t\right) =\mathbb{E}\left( X_{0}X_{t}\right) =\frac{\sigma
^{2}\Gamma \left( 2H+1\right) \sin \left( H\pi \right) }{2\pi }\int_{-\infty
}^{+\infty }\frac{e^{itx}\left\vert x\right\vert ^{1-2H}}{\lambda ^{2}+x^{2}}%
dx.
\end{equation*}%
But, due to (\ref{covfou}) $\rho \left( t\right) =\frac{\sigma
^{2}Hf_{H}\left( \lambda t\right) }{2\lambda ^{2H}}$, then we deduce that 
\begin{equation*}
f_{H}\left( \lambda t\right) =\frac{2\Gamma \left( 2H\right) \sin \left(
H\pi \right) \lambda ^{2H}}{H\pi }\int_{-\infty }^{+\infty }\frac{%
e^{itx}\left\vert x\right\vert ^{1-2H}}{\lambda ^{2}+x^{2}}dx.
\end{equation*}%
Finally, if we make the change of variable $x=\lambda v$ we obtain the
result.
\end{enumerate}
\end{proof}

\begin{proof}[Proof of Proposition 2]
\begin{equation*}
\mathbb{E}\left( X_{t}^{(1)}X_{s}^{(2)}\right) =\sigma ^{2}E\left(
\int_{-\infty }^{t}e^{-\lambda _{1}(t-u)}dB_{H}(u)\int_{-\infty
}^{s}e^{-\lambda _{2}(s-v)}dB_{H}(v)\right) .
\end{equation*}%
As $H>1/2,$ we can apply (\ref{pipiras}), then

\begin{equation*}
\mathbb{E}\left( X_{t}^{(1)}X_{s}^{(2)}\right) =\sigma
^{2}H(2H-1)\int_{-\infty }^{t}e^{-\lambda _{1}(t-u)}du\int_{-\infty
}^{s}e^{-\lambda _{2}(s-v)}\left\vert u-v\right\vert ^{2H-2}dv,
\end{equation*}%
now we make the change of variable: $w=t-u,$\ $z=s-v$ and we get that 
\begin{equation*}
\mathbb{E}\left( X_{t}^{(1)}X_{s}^{(2)}\right) =\sigma
^{2}H(2H-1)\int_{0}^{+\infty }e^{-\lambda _{1}w}dw\int_{0}^{+\infty
}e^{-\lambda _{2}z}\left\vert t-w+z-s\right\vert ^{2H-2}dz.
\end{equation*}%
Then, $\mathbb{E}\left( X_{t}^{(1)}X_{s}^{(2)}\right) $\ it depends on $t-s$,
so we only need to find a formula for $\mathbb{E}\left(
X_{0}^{(1)}X_{t}^{(2)}\right) .$ \[ \mathbb{E}\left(
X_{0}^{(1)}X_{t}^{(2)}\right) =\sigma ^{2}H(2H-1)\int_{0}^{+\infty
}dw\int_{0}^{+\infty }e^{-\lambda _{1}w-\lambda _{2}z}\left\vert
z-w-t\right\vert ^{2H-2}dz\] that after doing the change\ of variable $%
h=\lambda _{1}w+\lambda _{2}z$\ in the integral in $z$\ is equal to 
\begin{equation*}
\frac{\sigma ^{2}H(2H-1)}{\lambda _{2}}\int_{0}^{+\infty }dw\int_{\lambda
_{1}w}^{+\infty }e^{-h}\left\vert \frac{h-\lambda _{1}w}{\lambda _{2}}%
-w-t\right\vert ^{2H-2}dh=
\end{equation*}%
\begin{equation*}
\frac{\sigma ^{2}H(2H-1)}{\lambda _{2}^{2H-1}}\int_{0}^{+\infty
}dw\int_{\lambda _{1}w}^{+\infty }e^{-h}\left\vert h-\left( \lambda
_{1}+\lambda _{2}\right) w-\lambda _{2}t\right\vert ^{2H-2}dh=
\end{equation*}%
\begin{equation}
\frac{\sigma ^{2}H(2H-1)}{\lambda _{2}^{2H-1}}\int_{0}^{+\infty
}e^{-h}dh\int_{0}^{h/\lambda _{1}}\left\vert h-\left( \lambda _{1}+\lambda
_{2}\right) w-\lambda _{2}t\right\vert ^{2H-2}dw.  \label{2}
\end{equation}%
Now, we continue the calculus in the case $t\geq 0,$ and we separate in
zones according to the absolute value that apears in the last integral.
Then, we get \ that (\ref{2}) is equal to%
\begin{equation*}
\frac{\sigma ^{2}H(2H-1)}{\lambda _{2}^{2H-1}}\int_{0}^{\lambda
_{2}t}e^{-h}dh\int_{0}^{h/\lambda _{1}}\left( \left( \lambda _{1}+\lambda
_{2}\right) w+\lambda _{2}t-h\right) ^{2H-2}dw+
\end{equation*}%
\begin{equation*}
\frac{\sigma ^{2}H(2H-1)}{\lambda _{2}^{2H-1}}\int_{\lambda _{2}t}^{+\infty
}e^{-h}dh\int_{0}^{\frac{h-\lambda _{2}t}{\lambda _{1}+\lambda _{2}}}\left(
h-\left( \lambda _{1}+\lambda _{2}\right) w-\lambda _{2}t\right) ^{2H-2}dw+
\end{equation*}%
\begin{equation*}
\frac{\sigma ^{2}H(2H-1)}{\lambda _{2}^{2H-1}}\int_{\lambda _{2}t}^{+\infty
}e^{-h}dh\int_{\frac{h-\lambda _{2}t}{\lambda _{1}+\lambda _{2}}}^{h/\lambda
_{1}}\left( \left( \lambda _{1}+\lambda _{2}\right) w+\lambda _{2}t-h\right)
^{2H-2}dw.
\end{equation*}%
Now we make $s=\lambda _{2}t-h$ in the first summand and $s=h+\lambda _{1}t$%
\ in second, and we get%
\begin{equation*}
\frac{\sigma ^{2}H}{\lambda _{1}+\lambda _{2}}\left( e^{-\lambda
_{2}t}\Gamma \left( 2H\right) \lambda _{2}^{1-2H}-\lambda
_{2}^{1-2H}\int_{0}^{\lambda _{2}t}e^{-h}\left( \lambda _{2}t-h\right)
^{2H-1}dh\right) +
\end{equation*}%
\begin{equation*}
\frac{\sigma ^{2}H}{\lambda _{1}+\lambda _{2}}\lambda
_{1}^{1-2H}\int_{0}^{+\infty }e^{-h}\left( h+\lambda _{1}t\right) ^{2H-1}dh=
\end{equation*}%
\begin{equation*}
\frac{\sigma ^{2}H}{\lambda _{1}+\lambda _{2}}\left( e^{-\lambda
_{2}t}\Gamma \left( 2H\right) \lambda _{2}^{1-2H}-\lambda
_{2}^{1-2H}e^{-\lambda _{2}t}\int_{0}^{\lambda _{2}t}e^{s}s^{2H-1}ds\right) +
\end{equation*}%
\begin{equation*}
\frac{\sigma ^{2}H}{\lambda _{1}+\lambda _{2}}\lambda _{1}^{1-2H}e^{\lambda
_{1}t}\int_{\lambda _{1}t}^{+\infty }e^{-s}s^{2H-1}ds=
\end{equation*}%
\begin{equation*}
\frac{\sigma ^{2}H}{\lambda _{1}+\lambda _{2}}\left( \lambda
_{2}^{1-2H}f_{H}^{\left( 2\right) }\left( \lambda _{2}t\right) +\lambda
_{1}^{1-2H}\left[ e^{\lambda _{1}t}\Gamma \left( 2H\right) -e^{\lambda
_{1}t}\int_{0}^{\lambda _{1}t}e^{-s}s^{2H-1}ds\right] \right) =
\end{equation*}%
\begin{equation*}
\frac{\sigma ^{2}H}{\lambda _{1}+\lambda _{2}}\left( \lambda
_{2}^{1-2H}f_{H}^{\left( 2\right) }\left( \lambda _{2}t\right) +\lambda
_{1}^{1-2H}f_{H}^{\left( 1\right) }\left( \lambda _{1}t\right) \right) .
\end{equation*}%
In the case $t\leq 0,$ we work similarly.
\end{proof}

\begin{lemma}
\bigskip If $\lambda _{1},\lambda _{2},...,\lambda _{p}$ are positives reals
numbers, pairwise different, then%
\begin{equation*}
K_{i}+2\lambda _{i}\sum_{j\neq i}\frac{K_{j}}{\lambda _{i}+\lambda _{j}}=%
\frac{\lambda _{i}^{p-1}}{\prod\limits_{j\neq i}\left( \lambda _{i}+\lambda
_{j}\right) }\text{ for }i=1,2,3,...,p.
\end{equation*}
\end{lemma}

\begin{proof}
To obtain the result, it is enough\ to show that 
\begin{equation}
K_{1}+2\lambda _{1}\sum_{j=2}^{p}\frac{K_{j}}{\lambda _{i}+\lambda _{j}}=%
\frac{\lambda _{1}^{p-1}}{\prod\limits_{j=2}^{p}\left( \lambda _{1}+\lambda
_{j}\right) }.  \label{3}
\end{equation}%
Because, $K_{i}=\frac{\lambda _{i}^{p-1}}{\prod\limits_{j\neq i}^{p}\left(
\lambda _{i}-\lambda _{j}\right) },$ then (\ref{3}) is equal to 
\begin{equation*}
\frac{\lambda _{1}^{p-1}}{\prod\limits_{j=2}^{p}\left( \lambda _{1}-\lambda
_{j}\right) }+2\lambda _{1}\sum_{j=2}^{p}\frac{K_{j}}{\lambda _{1}+\lambda
_{j}}=\frac{\lambda _{1}^{p-1}}{\prod\limits_{j=2}^{p}\left( \lambda
_{1}+\lambda _{j}\right) }
\end{equation*}%
wich is equivalent to prove that (if we call $x=\lambda _{1}$) 
\begin{equation}
\frac{x^{p-2}}{\prod\limits_{j=2}^{p}\left( x-\lambda _{j}\right) }-\frac{%
x^{p-2}}{\prod\limits_{j=2}^{p}\left( x+\lambda _{j}\right) }%
=-2\sum_{j=2}^{p}\frac{K_{j}}{x+\lambda _{j}}.  \label{4}
\end{equation}%
In fact, we develop the quotient in simple fractions, and we obtain that (%
\ref{4}) is equal to 
\begin{equation*}
\frac{x^{p-2}}{\prod\limits_{j=2}^{p}\left( x-\lambda _{j}\right) }-\frac{%
x^{p-2}}{\prod\limits_{j=2}^{p}\left( x+\lambda _{j}\right) }%
=\sum_{i=2}^{p}\left( \frac{A_{i}}{x-\lambda _{i}}-\frac{B_{i}}{x+\lambda
_{i}}\right) =
\end{equation*}%
\begin{equation*}
\sum_{i=2}^{p}\left( \frac{\lambda _{i}^{p-2}}{\prod\limits_{j\neq i}\left(
\lambda _{i}-\lambda _{j}\right) }\frac{1}{x-\lambda _{i}}-\frac{\left(
-\lambda _{i}\right) ^{p-2}}{\prod\limits_{j\neq i}\left( \lambda
_{j}-\lambda _{i}\right) }\frac{1}{x+\lambda _{i}}\right) =
\end{equation*}%
\begin{equation*}
\sum_{i=2}^{p}\left( \frac{\lambda _{i}^{p-2}2\lambda _{i}}{%
\prod\limits_{j\neq i}\left( \lambda _{i}-\lambda _{j}\right) }\frac{1}{%
\left( x-\lambda _{i}\right) \left( x+\lambda _{i}\right) }\right) =
\end{equation*}%
\begin{equation*}
-2\sum_{i=2}^{p}\frac{\lambda _{i}^{p-1}}{\left( \lambda _{i}-x\right)
\prod\limits_{j\neq i}\left( \lambda _{i}-\lambda _{j}\right) }\frac{1}{%
\left( x+\lambda _{i}\right) }=-2\sum_{j=2}^{p}\frac{K_{j}}{x+\lambda _{j}}.
\end{equation*}
\end{proof}

\begin{lemma}
If $\lambda _{1},\lambda _{2},...,\lambda _{p}$ son positives real numbers,
pairwise different, then 
\begin{equation*}
\frac{x^{2p-2}}{\prod\limits_{i=1}^{p}\left( x^{2}-\lambda _{j}^{2}\right) }%
=\sum_{i=1}^{p}\frac{\lambda _{i}^{2p-2}}{\prod\limits_{j\neq i}\left(
\lambda _{i}^{2}-\lambda _{j}^{2}\right) }\frac{1}{x^{2}-\lambda _{i}^{2}}.
\end{equation*}
\end{lemma}

\begin{proof}
We decompose in simple fractions, and we obtain 
\begin{equation*}
\frac{x^{2p-2}}{\prod\limits_{i=1}^{p}\left( x^{2}-\lambda _{j}^{2}\right) }%
=\frac{x^{2p-2}}{\prod\limits_{i=1}^{p}\left( x-\lambda _{j}\right) \left(
x+\lambda _{j}\right) }=\sum_{i=1}^{p}\left( \frac{A_{i}}{x-\lambda _{i}}+%
\frac{B_{i}}{x+\lambda _{i}}\right) =
\end{equation*}%
\begin{equation*}
\sum_{i=1}^{p}\left( \frac{\lambda _{i}^{2p-2}}{2\lambda _{i}\prod_{j\neq
i}\left( \lambda _{i}^{2}-\lambda _{j}^{2}\right) }\frac{1}{x-\lambda _{i}}-%
\frac{\lambda _{i}^{2p-2}}{2\lambda _{i}\prod_{j\neq i}\left( \lambda
_{i}^{2}-\lambda _{j}^{2}\right) }\frac{1}{x+\lambda _{i}}\right) =
\end{equation*}%
\begin{equation*}
\sum_{i=1}^{p}\frac{\lambda _{i}^{2p-2}}{2\lambda _{i}\prod_{j\neq i}\left(
\lambda _{i}^{2}-\lambda _{j}^{2}\right) }\left( \frac{1}{x-\lambda _{i}}-%
\frac{1}{x+\lambda _{i}}\right) =\sum_{i=1}^{p}\frac{\lambda _{i}^{2p-2}}{%
\prod_{j\neq i}\left( \lambda _{i}^{2}-\lambda _{j}^{2}\right) }\left( 
\frac{1}{x^{2}-\lambda _{i}^{2}}\right) .
\end{equation*}
\end{proof}

\begin{proof}[Proof of Proposition 3]
We start with (\ref{covgeneral}) in the case $p_{1}=p_{2}=......=p_{q}=1,$ 
\begin{equation*}
\gamma (t)=\mathbb{E}\left( X_{t}X_{0}\right)
=\sum_{i,j=1}^{q}K_{i}K_{j}\gamma _{\lambda _{h},\lambda _{h^{\prime
}}}^{\left( 0,0\right) }(t)
\end{equation*}%
and using (\ref{gamma}) 
\begin{equation*}
\gamma (t)=\sigma ^{2}H\sum_{i,j=1}^{q}K_{i}K_{j}\frac{\left( \lambda
_{i}^{1-2H}f_{H}^{\left( 1\right) }(\lambda _{i}t)+\lambda
_{j}^{1-2H}f_{H}^{\left( 2\right) }(\lambda _{j}t)\right) }{\lambda
_{i}+\lambda _{j}}=
\end{equation*}%
\begin{equation*}
\sigma ^{2}H\left( \sum_{i,j=1}^{p}\frac{K_{i}^{2}}{2\lambda _{i}}\lambda
_{i}^{1-2H}f_{H}(\lambda _{i}t)+\sum_{i=1}^{p}K_{i}\lambda
_{i}^{1-2H}f_{H}^{\left( 1\right) }(\lambda _{i}t)\sum_{j\neq i}\frac{K_{j}}{%
\lambda _{i}+\lambda _{j}}\right) +
\end{equation*}%
\begin{equation*}
\sigma ^{2}H\sum_{i=1}^{p}K_{i}\lambda _{i}^{1-2H}f_{H}^{\left( 2\right)
}(\lambda _{i}t)\sum_{j\neq i}\frac{K_{j}}{\lambda _{i}+\lambda _{j}}=
\end{equation*}%
\begin{equation*}
\sigma ^{2}H\sum_{i=1}^{p}K_{i}\lambda _{i}^{-2H}f_{H}(\lambda _{i}t)\left( 
\frac{K_{i}}{2}+\lambda _{i}\sum_{j\neq i}\frac{K_{j}}{\lambda _{i}+\lambda
_{j}}\right)
\end{equation*}%
Now, using Lemma 1, we obtain that the last expression is
equal to 
\begin{equation*}
\frac{\sigma ^{2}H}{2}\sum_{i=1}^{p}K_{i}\lambda _{i}^{-2H}f_{H}(\lambda
_{i}t)\frac{\lambda _{i}^{p-1}}{\prod\limits_{j\neq i}\left( \lambda
_{i}+\lambda _{j}\right) }=
\end{equation*}%
\begin{equation*}
\frac{\sigma ^{2}H}{2}\sum_{i=1}^{p}\frac{\lambda _{i}^{p-1}}{%
\prod\limits_{j\neq i}\left( \lambda _{i}-\lambda _{j}\right) }\lambda
_{i}^{-2H}f_{H}(\lambda _{i}t)\frac{\lambda _{i}^{p-1}}{\prod\limits_{j\neq
i}\left( \lambda _{i}+\lambda _{j}\right) }=
\end{equation*}%
\begin{equation*}
\frac{\sigma ^{2}H}{2}\sum_{i=1}^{p}\frac{\lambda _{i}^{2p-2H-2}}{%
\prod\limits_{j\neq i}\left( \lambda _{i}^{2}-\lambda _{j}^{2}\right) }%
f_{H}(\lambda _{i}t).
\end{equation*}
\end{proof}

\begin{lemma}
If $\lambda _{1},\lambda _{2},...,\lambda _{p}$ are positive reals numbers
pairwise different, then 
\begin{equation*}
\frac{x^{2p-2}}{\prod\limits_{i=1}^{p}\left( \lambda _{i}^{2}+x^{2}\right) }%
=\sum_{i=1}^{p}\frac{\lambda _{i}^{2p-2}}{\prod\limits_{j\neq i}\left(
\lambda _{i}^{2}-\lambda _{j}^{2}\right) }\frac{1}{\lambda _{i}^{2}+x^{2}}.
\end{equation*}
\end{lemma}
\begin{proof}

We will proceed by induction in $p.$ For $p=1,$ the equality is evident.
Supose that the equality holds por $p.$ Then, calculate 
\begin{equation}
\frac{x^{2p}}{\prod\limits_{i=1}^{p+1}\left( \lambda _{i}^{2}+x^{2}\right) }%
=\frac{x^{2p-2}}{\prod\limits_{i=1}^{p}\left( \lambda _{i}^{2}+x^{2}\right) 
}\frac{x^{2}}{\lambda _{p+1}^{2}+x^{2}}  \label{p}
\end{equation}
applying the hypothesis of induction, we deduce that( \ref{p}) is equal to 
\begin{equation*}
\sum_{i=1}^{p}\frac{\lambda _{i}^{2p-2}}{\prod\limits_{j\neq i}\left(
\lambda _{i}^{2}-\lambda _{j}^{2}\right) }\frac{1}{\lambda _{i}^{2}+x^{2}}%
\frac{x^{2}}{\lambda _{p+1}^{2}+x^{2}}=
\end{equation*}%
\begin{equation*}
\sum_{i=1}^{p}\frac{\lambda _{i}^{2p-2}}{\left( \lambda _{i}^{2}-\lambda
_{p+1}^{2}\right) \prod\limits_{j\neq i}^{p}\left( \lambda _{i}^{2}-\lambda
_{j}^{2}\right) }\left( \frac{\lambda _{i}^{2}}{\lambda _{i}^{2}+x^{2}}-%
\frac{\lambda _{p+1}^{2}}{\lambda _{p+1}^{2}+x^{2}}\right) =
\end{equation*}%
\begin{equation*}
\sum_{i=1}^{p}\frac{\lambda _{i}^{2p-2}}{\left( \lambda _{i}^{2}-\lambda
_{p+1}^{2}\right) \prod\limits_{j\neq i}^{p}\left( \lambda _{i}^{2}-\lambda
_{j}^{2}\right) }\left( \frac{\lambda _{i}^{2}}{\lambda _{i}^{2}+x^{2}}-%
\frac{\lambda _{p+1}^{2}}{\lambda _{p+1}^{2}+x^{2}}\right) =
\end{equation*}%
\begin{equation}
\sum_{i=1}^{p}\frac{\lambda _{i}^{2p}}{\prod\limits_{j\neq i}^{p+1}\left(
\lambda _{i}^{2}-\lambda _{j}^{2}\right) }\frac{1}{\lambda _{i}^{2}+x^{2}}%
-\sum_{i=1}^{p}\frac{\lambda _{p+1}^{2}\lambda _{i}^{2p-2}}{%
\prod\limits_{j\neq i}^{p+1}\left( \lambda _{i}^{2}-\lambda _{j}^{2}\right) 
}\frac{1}{\lambda _{p+1}^{2}+x^{2}}.  \label{induc}
\end{equation}%
Now, using Lemma 2 with $x=\lambda _{p+1}$, we obtain that 
\begin{equation*}
\sum_{i=1}^{p}\frac{\lambda _{i}^{2p-2}}{\prod\limits_{j\neq i}^{p+1}\left(
\lambda _{i}^{2}-\lambda _{j}^{2}\right) }=-\frac{\lambda _{p+1}^{2p-2}}{%
\prod\limits_{j=1}^{p}\left( \lambda _{p+1}^{2}-\lambda _{j}^{2}\right) }
\end{equation*}%
and then (\ref{induc}) is equal to 
\begin{equation*}
\sum_{i=1}^{p+1}\frac{\lambda _{i}^{2p}}{\prod\limits_{j\neq i}^{p+1}\left(
\lambda _{i}^{2}-\lambda _{j}^{2}\right) }\frac{1}{\lambda _{i}^{2}+x^{2}}.
\end{equation*}
\end{proof}

\begin{proof}[Proof of Theorem 2]
First, we will prove the result for the case in wich
$\left\{ X_{t}\right\}_{t\in \mathbb{R}} \sim $ FOU$\left( \lambda _{1},\lambda _{2},...,\lambda
_{p},\sigma ,H\right) .$ Using property (5) of $f_{H}$ in (\ref{covfoup}) and Lemma 3,
we obtain that $\mathbb{E}\left( X_{0}X_{t}\right) =$

\begin{equation*}
\frac{\sigma ^{2}H\Gamma (2H+1)\sin \left( H\pi \right) }{2\pi }%
\int_{-\infty }^{+\infty }e^{itx}\left\vert x\right\vert
^{1-2H}\sum_{i=1}^{p}\frac{\lambda _{i}^{2p-2}}{\prod\limits_{j\neq
i}\left( \lambda _{i}^{2}-\lambda _{j}^{2}\right) }\frac{1}{\lambda
_{i}^{2}+x^{2}}dx=
\end{equation*}
\begin{equation*}
\frac{\sigma ^{2}H\Gamma (2H+1)\sin \left( H\pi \right) }{2\pi }%
\int_{-\infty }^{+\infty }e^{itx}\frac{|x|^{2p-2H-1}}{\prod_{i=1}^{p}\left( \lambda _{i}^{2}+x^{2}\right )}dx,
\end{equation*}
then (\ref{espectraldistintos}) it holds. 

Now, \ in the general case in wich $\left\{ X_{t}\right\} _{t\in \mathbb{R}%
}\sim $FOU$\left( \lambda _{1}^{\left( p_{1}\right) },\lambda _{2}^{\left(
p_{2}\right) },...,\lambda _{q}^{\left( p_{q}\right) },\sigma ,H\right) $
where $p_{1}+p_{2}+...+p_{q}=p$, observe that the process $\left\{
X_{t}\right\} _{t\in \mathbb{R}}$ is puntual limit of the processes of $%
\left\{ X_{t}^{\left( n\right) }\right\} _{t\in \mathbb{R}}\sim $FOU$\left(
p\right) $ with parameters $\lambda _{1},\lambda _{1}+1/n,...,\lambda
_{1}+(p_{1}-1)/n,....,\lambda _{q},\lambda _{q}+1/n,...,\lambda
_{q}+(p_{q}-1)/n,\sigma ,H$. Also, the spectral density of $\left\{ X^{(n)}_{t}\right\} _{t\in \mathbb{R}%
} $ it is according to (\ref{espectraldistintos}). Now, using that the processes are Gaussian, we
deduce that $f^{\left( X^{\left( n\right) }\right) }(x)\rightarrow f^{\left(
X\right) }(x)$ for all $x$, then (\ref{espectral}) holds.
\end{proof}

\section{Conclusions}
In this work we have presented a Gaussian processes that arises from
the iteration of p fractional Ornstein-Uhlenbeck processes generated by the same fractional Brownian motion.
When the values of $\lambda_{i}$ are pairwise different, this iteration results in a particular linear combination  of each 
fractional Ornstein-Uhlenbeck process. We proved that when $H>1/2$ and $\lambda_{i}$ are pairwise
different, the auto-covariance function of the process can be expressed as a linear combination of auto-covariance function of each 
FOU$(\lambda_{i},\sigma,H)$.
We have obtained a explicit formula for the spectral density of the process that  allows us to deduce that,  
although every fractional Ornstein-Uhlenebeck process with $H>1/2$ is a long memory process, 
for $p\geq 2$ the iteration results in a short memory process.
We adjusted these processes to model three real time data sets, and compare their predictive performance with respect 
to ARMA models. In all three cases, similar or better performances were observed. They were observed even
the good performance of FOU$(\lambda^{(i)},\sigma,H)$ for different values of $i$ in each case.  To estimate the parameters of FOU$(p)$
we use a naive method that consists in matching correlations in a certain number of points.

\section{Acknowledgement}
I want to thank Enrique Cabaña and Jorge Graneri for their generosity and contribution of ideas for this work.

\section{References}

\begin{enumerate}

\item Arratia, A., Caba\~{n}a, A. \& Caba\~{n}a, E. ``A construction of Continuous
time ARMA models by iterations of Ornstein-Uhlenbeck process", \textit{SORT} Vol 40 (2) 267-302, 2016.

\item Box, G. E. P., Jenkins, G. M. \& Reinsel, G. C. ``Time Series Analysis,
Forecasting and Control", Prentice Hall, 1994.

\item Brockwell, P. J., Davies, R. A., ``Introduction to Time Series and
Forecasting", \textit{Springer}, 2002.

\item Cheridito, P., Kawaguchi, H. \& Maejima, M. ``Fractional Ornstein-Uhlenbeck
Processes". \textit{Electronic Journal of Probability}, 8(3): 1-14,2003.

\item Cleveland, W. S., ``The inverse autocorrelations of a time series and their
applicattions". Technometrics, 14: 277-298, 1971.

\item Langevin, P., ``Sur la te'eorie du mouvement brownien", \textit{C.R. Acad. Sci.} Par%
\'{\i}s 146, 530-533, 1908.

\item McLeod, A. I. \& Zang, Y. ``Partial autocorrelation parametrization for
subset autoregression". \textit{J of Time Series Analysis}, 27: 599-612, 2006.

\item Pipiras, V. and Taqqu, M. ``Integration questions related to fractional
Brownian motion", \textit{Prob. Th. Rel. Fields}, 118, 121-291,2000.

\item Uhlenbeck, G.E. \& Ornstein, L.S., ``On the theory of Brownian Motion";
\textit{Physical Review} 36, 823-841, 1930.

\item Willmott, C. J., ``Some comments of the evaluation of model performance".
\textit{Buletin of the American Meteorological Society}, 63, 1309-1313, 1982.

\end{enumerate}

\end{document}